\documentclass[hidelinks,onefignum,onetabnum]{siamart250211}

\usepackage{amsfonts,amsmath,amssymb}
\usepackage{graphicx}
\usepackage{epstopdf}
\usepackage{algorithmic}
\usepackage{cleveref}
\usepackage{mathtools}
\usepackage{comment}
\usepackage{todonotes}

\usepackage{lipsum}
\ifpdf
  \DeclareGraphicsExtensions{.eps,.pdf,.png,.jpg}
\else
  \DeclareGraphicsExtensions{.eps}
\fi

\usepackage{enumitem}
\setlist[enumerate]{leftmargin=.5in}
\setlist[itemize]{leftmargin=.5in}

\newcommand{\BYDEF}{\mathrel{\overset{\makebox[0pt]{\textup{\tiny def}}}{=}}}
\def\a{\alpha}
\def\b{\beta}
\def\d{\delta}
\def\e{\epsilon}

\def\l{\lambda}
\def\ph{\varphi}

\def\lo{\scalebox{1}{\L}}

\def\O{\cl F}
\def\D{{\mathcal D}}
\def\Cp{{\mathcal C_+}}
\def\Mp{{\mathcal M_+}}

\def\NN{\mathbb{N}}
\def\RR{\mathbb{R}}
\def\rng{\mathcal{R}}
\def\drng{\cl D}
\def\pos{\mathrm{\cl Pos}}
\def\bnd{\flat} 
\newcommand{\norm}[1]{\|{#1}\|}
\newcommand{\cl}[1]{\mathcal{#1}}
\newcommand{\bm}[1]{\boldsymbol{#1}}
\newcommand{\vspan}[1]{\langle{#1}\rangle}
\newcommand{\apply}[2]{\langle {#1}, {#2} \rangle}
\newcommand{\mult}[2]{ {#1} \cdot {#2}}
\def\dop{\mathfrak{d}}
\def\pop{\mathfrak{p}}

\newcommand*{\QED}[1][$\square$]{%
\leavevmode\unskip\penalty9999 \hbox{}\nobreak\hfill
    \quad\hbox{#1}%
}

\def\brev{\color{black}}
\def\erev{\color{black}}

\newsiamremark{remark}{Remark}
\newsiamremark{hypothesis}{Hypothesis}
\crefname{hypothesis}{Hypothesis}{Hypotheses}
\newsiamthm{claim}{Claim}

\newsiamremark{assumption}{Assumption}
\newsiamremark{condition}{Condition}
\newsiamremark{example}{Example}

\headers{The Effective Countable Generalized Moment Problem}{L. Gamertsfelder and B. Mourrain}

\title{The Effective Countable Generalized Moment Problem
}

\author{Lucas Gamertsfelder\thanks{Centre Inria of Université Côte d'Azur, 2004 Route des Lucioles, 06902 Valbonne, France 
  (\email{Lucas.Gamertsfelder@inria.fr}).}
\and Bernard Mourrain\thanks{Centre Inria of Université Côte d'Azur, 2004 Route des Lucioles, 06902 Valbonne, France  
  (\email{Bernard.Mourrain@inria.fr}).} }

\usepackage{amsopn}

\ifpdf
\hypersetup{
  pdftitle={The Effective Countable Generalized Moment Problem},
  pdfauthor={L. Gamertsfelder and B. Mourrain}
}
\fi

\begin{document}

\maketitle

\begin{abstract}
We establish new convergence rates for the Moment-Sum-of-Squares (Moment-SoS) relaxations for the Generalized Moment Problem (GMP) with countable moment constraints on vectors of measures,  under dual optimum attainment, $S$-fullness and Archimedean conditions. These bounds, which adapt to the geometry of the underlying semi-algebraic set, apply to both the convergence of optima, and to the convergence in Hausdorff distance between the relaxation feasibility set and the GMP feasibility set. 
\brev
We show that under the previous conditions, the sequence of optimizers of the relaxations converge to the optimizer of the GMP for the weak$^*$ topology, provided this optimal measure is unique.
\erev
This research provides quantitative geometry-adaptive rates for GMPs cast as linear programs on measures. It complements earlier analyses of specific GMP instances (e.g., polynomial optimization) as well as recent methodological frameworks that have been applied to volume computation and optimal control. 
We apply the convergence rate analysis to symmetric tensor decomposition problems, providing new effective error bounds for the convergence of the Moment-SoS hierarchies for tensor decomposition.
\end{abstract}

\begin{keywords}
Generalized moment problems, Moment-SoS relaxations, Positivstellensatz, convergence rates, symmetric tensor decomposition
\end{keywords}

\begin{MSCcodes}
44A60, 90C22, 15A69, 14P10 
\end{MSCcodes}

The Generalized Moment Problem (GMP) extends classical moment problems into the framework of conic linear programming over positive Borel measures. 
Building on moment problems such as those of Hamburger, Stieltjes, and Hausdorff, the GMP incorporates an optimization: finding measures that both satisfy moment constraints and minimize a given linear functional.
With applications in numerous fields including control, signal processing, and statistics (see e.g. \cite{lasserre2018moment}), moment problems remain an active area of research since their introduction in the 19th century (see \cite{akhiezer2020classical, landau1987moments, schmudgen2017} 
for a historical overview). 

Moment-Sum-of-Squares (SoS) relaxations provide a systematic approach to solving the GMP 
\cite{lasserre2001global, lasserre2018moment}. The method proceeds in two  steps: first relaxing the cone of positive Borel measures to non-negative linear functionals over an infinite-dimensional quadratic module; a relaxation that is exact under the Archimedean condition via Putinar's Positivstellensatz \cite{putinar1993positive}. Second, truncating the quadratic module's degree to yield a hierarchy of semi-definite programs (SDPs), with accuracy improving as the truncation degree increases. This approach has been successfully applied to Polynomial Optimization Problems (POP), which are a special class of moment problems. 
While the convergence of Moment-SoS relaxations for POP was proven (under an Archimedean condition) in \cite{lasserre2001global}, only recent contributions provide polynomial convergence rates \cite{Fang2020,laurent2023effective,baldi2023effective,BaldiSlot2024,baldi2024lojasiewicz}.

For the GMP, Moment-SoS relaxations were first investigated in \cite{lasserre2008semi}, where qualitative convergence of the optimal values was established under Archimedean and strict positivity conditions. 
The analysis of Moment-SoS hierarchies for GMPs on a single positive measure was further investigated e.g. in
\cite{curto1996solution,CurtoFialkow2005}
for the solution of GMP via flat truncation, in \cite{Nie2014,HuangNieYuan24} regarding the uniqueness of the maximizer and the exactness of the relaxation under feasibility  and regularity conditions, 
or in \cite{Laurent2009} via sparse flat extension properties.
For an extended overview on moment problems and positive polynomials, we refer to \cite{schmudgen2017}.

The study of convergence rates for the Moment-SoS hierarchy in the GMP context has also advanced. For specific applications, such as Polynomial Optimization Problems (where there is a single moment equality constraint) \cite{schweighofer_complexity_2004,nie_complexity_2007,klerk2019survey,Fang2020,laurent2023effective,baldi2023effective,BaldiSlot2024}, optimal control problems (with moment equality constraints on a single occupational measure) \cite{korda2017convergence}, and volume computation (exploiting Stokes equations for moment constraints) \cite{tacchi2023stokes}, convergence rates have been analyzed.
More recently, \cite{schlosser2024convergence} introduced a methodological framework to compute bounds on convergence rates of the optima for GMPs involving vectors of measures with moment equality constraints, leveraging an Effective Putinar's Positivstellensatz \cite{baldi2023effective}. This framework was applied to derive polynomial convergence rates for applications including optimal control and volume computation.
However, explicit rates for general cases of GMPs were not established.

In this paper, we establish convergence rates for both the optima and feasibility sets for GMPs with countable moment constraints on vectors of measures. We define a linear program on measures as one where the unknown is a vector $\mu= [\mu_1, \ldots, \mu_m]$ of positive measures $\mu_i$ supported on compact sets $S_i\subset \RR^{n_i}$.
Where $J$ is a countable index set, the constraints are of the form $t - \apply{\mu}{h} \in K$ where $K$ is a cone in $\RR^{J}$, $t\in \RR^{J}$, $h=(h_{i,j})\in (\prod_{i=1}^{m} \cl C_i)^{J}$ is a matrix of continuous functions $h_{i,j}\in \cl C_i$ on $S_i$ and $\apply{\mu}{h}= [\sum_{i=1}^{m} \int_{S_i} h_{i,j} d\mu_i, \ j \in J]$. The objective function is of the form $\apply{\mu}{f}=\sum_{i=1}^{m} \int_{S_i} f_i d\mu_i$ where $f= (f_1, \ldots, f_m) \in \prod_{i=1}^{m} \cl C_i$. This linear program on measures is approximated by a hierarchy of Moment-SoS convex relaxations (also known as Lasserre's relaxation).

We establish the convergence of these hierarchies under dual optimum attainment, $S$-fullness and Archimedean conditions, and provide polynomial convergence rates that adapt to the underlying semi-algebraic set geometry. 
Our results apply both to the convergence of optima and to the convergence of the moment feasibility sets. 
Our main results are the following.\\
\noindent{\textbf{Theorem A.}} (See  \Cref{cor:gmp_convergence}) {\it\ Under  \Cref{as:x-full}, \Cref{as:dual_optimizer} and  \Cref{as:archimedean},
$$
0 \le \pop - \pop_{\ell} \le \dop -\dop_{\ell} \le \kappa \,\ell^{-\theta} 
$$
where $\pop$, $\dop$ (resp. $\pop_\ell$, $\dop_\ell$) are the optima of the (resp. relaxation of the) primal and dual Generalized Moment Problem \eqref{eq:gmp_obj} and \eqref{eq:gmp_dual_obj}  (resp. at order $\ell\in \NN$), and $\kappa, \, \theta$
are constants depending explicitly on the input data of the GMP.
}\\

\noindent{\textbf{Theorem B}.} (See \Cref{thm:haus_dis})  {\it\ Under  \Cref{as:x-full}, \Cref{as:dual_optimizer} and  \Cref{as:archimedean}, for $\ell \ge k$, 
$$
    \rho_H(\O^{(k)}, L^{(k)}_\ell) \le  \kappa' \, \ell^{-\theta},
$$
where $\O^{(k)}$, $L_{\ell}^{(k)}$ are respectively the cones of 
moment sequences truncated at order $k$ of positive measures and truncated linear functionals positive on the degree $\ell$ truncated quadratic module, and $\rho_H$ is the Hausdorff distance.
}

These results establish polynomial rates of convergence for GMPs, building on the recent progress on the Effective Positivstellensatz (see \Cref{thm:putinar_effective}). 

\brev
As a consequence, we deduce (see \Cref{thm:unique opt})  that 
under Assumptions \ref{as:x-full}, \ref{as:dual_optimizer} and  \ref{as:archimedean}, the sequence of optimizers of the relaxations converge to the optimizer of the GMP, provided there is a unique optimal measure of the GMP.
\erev

We apply this analysis to minimal symmetric tensor decomposition (\Cref{sec:applications}).
Symmetric tensors play important roles in independent component analysis, telecommunications, and psychometrics; see \cite{kolda2009tensor, brachat2010symmetric} for overviews.
Additive decompositions of a tensor provide a convenient representation, and minimal decompositions are particularly valuable \cite{mourrain2020minimal}.
Current algorithms can recover symmetric decompositions \cite{brachat2010symmetric}, and minimal forms for low-rank tensors \cite{mourrain2020minimal}. However, high-rank tensor decompositions remain challenging. We propose an optimization approach to this problem, formulating the minimal decomposition as a GMP. Our approach leverages the homogeneous polynomial representation of symmetric tensors, and the \textit{apolar product}.
A homogeneous polynomial admits a Waring decomposition when its dual linear functional with respect to the apolar product can be represented as a finite sum of weighted Dirac measures. The GMP formulation is based upon finding such a minimal representation.
We prove that these GMP problems have no duality gap, and  we provide new results of attainment of optima and new convergence rates (\Cref{thm:tensordec_pos}, \Cref{thm:tensordec_real}).

The paper is structured as follows.
In \Cref{sec:GMP}, we formally introduce the GMP and its dual formulation. We present the $S$-fullness condition, which ensures both strong duality between the primal and dual formulations, and compactness in the primal. Additionally, we verify that the $S$-fullness condition holds in the case of the POP.
In \Cref{sec:moment-SOS}, we introduce the Moment-Sum-of-Squares (SoS) relaxations of the GMP and analyze their rate of convergence under the $S$-fullness condition and Archimedean condition. We prove our first main result, \Cref{cor:gmp_convergence}, providing bounds on the difference in value between the GMP and its Moment-SoS relaxation. We also prove the second main result, \Cref{thm:haus_dis}, providing bounds on the Hausdorff distance between the GMP cone and the relaxed cones.
Finally, in \Cref{sec:applications}, we investigate symmetric tensor decomposition problems and formulate them as GMPs. We verify that the $S$-fullness condition holds and state the convergence rates of the moment relaxations using the main results.
We provide examples illustrating the numerical behavior of the approach.

\section{The Generalized Moment Problem}\label{sec:GMP}

Let $S\subset \RR^n$ be a compact subset of $\RR^n$. We denote by $\cl C(S)$ the set of continuous functions on $S$.
It is a Banach space when endowed with the sup-norm: for $q\in \cl C(S)$, $\|q\|_{\infty}= \max_{x\in S} |q(x)|$.
We denote by $\cl C_{+}(S)$ the convex cone of continuous functions $f\in \cl C(S)$ which are positive on $S$: $\forall x\in S, f(x) \ge 0$.

We denote by $\cl M(S)= \cl C(S)^{*}$ the space of continuous linear functionals on $\cl C(S)$, equipped with the weak$^*$ topology and the operator norm: for $\lambda \in \cl M(S)$, $\|\lambda\|_{\infty} = \max_{\|q\|_{\infty}\le 1} \apply{\lambda}{q}$. In the weak$^*$ topology, a sequence $\{\lambda_n\} \subset \cl M(S)$ converges to $\lambda \in \cl M(S)$ if and only if
\begin{equation}
\forall \, q \in \cl C(S), \quad \apply{\lambda_n}{q} \to \apply{\lambda}{q}.
\end{equation}
By \cite[Theorem 4.3.13]{ash2014measure}, $\cl M(S)$ is the set of finite signed Borel measures on $S$, so that $\forall \mu \in \cl M(S)$,  $\int_S q \, d\mu \BYDEF \apply{\mu}{q}$.

Note that when the space of measures $\cl M(S)$ is endowed with the norm of total variation (equivalently the operator norm on the linear functional representation \cite[p. 186]{ash2014measure}) it is a Banach space.

Let $\cl M_+(S)$ be its positive cone, i.e., the space of positive Borel measures on $S$.
%
%
%
By Riesz-Haviland Theorem (see \cite[Theorem 2.14]{rudin1987real}), 
for any $\mu\in \cl M(S)$, $\mu \in \Mp(S)$ if and only if  $\forall p \in \cl C_{+}(S), \apply{\mu}{p}\ge 0$.


Let $m$ be a non-zero natural number. For $i=1,\ldots, m$, let $S_i \subset \RR^{n_i}$ be a compact subset of $\RR^{n_i}$ for $n_i \in \mathbb{N}$. Where $\cl C(S_i)$, $\cl M (S_i)$, $\Mp(S_i)$, $C_+(S_i)$ are as above, define $\cl C = \prod_{i=1}^m \cl C(S_i)$, $\cl M = \prod_{i=1}^m \cl M (S_i)$, $\Mp = \prod_{i=1}^m \Mp(S_i)$, $\cl C_+ = \prod_{i=1}^m \cl C_+(S_i)$.


Let $f \in \cl C$ and $J$ be a countable index set. For $j \in J$, let $t_j \in \RR$ and $h_j \in \cl C$. Set $t = (t_j)_{j \in J}$, $h = (h_j)_{j \in J} = (h_{i,j})_{1 \le i\le m, j \in J} \in \cl C^{J}$ and $K\subset \RR_+^J$ a closed convex cone in the positive orthant of $\RR^J$, the space of real sequences indexed by $J$ \brev endowed with the product topology\erev. We equip $\cl C$ and $\cl M$ with the product topologies: sequences in the former sets converge if and only if each component converges independently.
We consider the \emph{$m$-dimensional Generalized Moment Problem} ($m$-GMP):
For $f \in \cl C$, $h_j \in \cl C$ for all $j \in J$,
\begin{equation}\label{eq:gmp_obj}
        \begin{split}
            \pop^* \BYDEF\ \inf\ \ &\apply{\mu}{f}\\
            \textnormal{s.t.} \ \
            &t -\apply{\mu}{h} \in K\\
            &\mu \in \mathcal{M}_+,
        \end{split}
\end{equation}
where 
$$
\apply{\mu}{f} \BYDEF \sum_{i=1}^m\apply{\mu_i}{f_i}, \ \ \ \text{and} \ \ \ 
\apply{\mu}{h} \BYDEF \left(\sum_{i=1}^m\apply{\mu_i}{h_{i,j}}\right)_{j\in J}
$$
The feasible set of $\pop^*$ has the form
$$
\Omega \BYDEF \left\{  \mu \in \Mp \mid t - \apply{\mu}{h} \in K \right\}.
$$
We consider also the \emph{$m$-dimensional Dual Generalized Moment Problem} ($m$-DGMP):
\begin{equation}\label{eq:gmp_dual_obj}
        \begin{split}
            \dop^* = \ \textnormal{sup} \ \ & \mult{t}{v} \\
            \textnormal{s.t.} \ \ &f - \mult{h}{v} \in \Cp, \\
            &v \in - K^{*}.
        \end{split}
\end{equation}
where $\mult{h}{v}=\sum_{j \in J} h_j\, v_j$ and $K^*$ is the dual cone of $K$, defined as $K^{*}=\{v \in \RR_0^{J}: \forall u\in K, \mult{u}{v} \ge 0\}$, with $\RR_0^J$ being the space of sequences of finite support.

\begin{example}
    When $m=1$, $N=1$, $h=1\in \cl C(S)$, $t=1 \in \RR$ and $K=\{0\}$, we recover the  classical polynomial optimization problem 
\begin{equation} \label{eq:opt}
\begin{array}{rlrl}
    f_{\min}\ =\ \inf_{x\in S} f(x)\  = \ \inf\ \   & \apply{\mu}{f}  & = \ \sup\  & v\\
             \ \textnormal{s.t.} \ & 1-\apply{\mu}{1} = 0  & \ \textnormal{s.t.} \  
            & f - v\in \cl C_+(S)  \\
            &\mu \in \mathcal{M}_+ & & v \in \RR.
\end{array}
\end{equation}
\end{example}

    The following standard assumption (see e.g. \cite[Lemma 1]{lasserre2008semi}), which ensures the existence of at least one measure satisfying all moment constraints \emph{holds throughout this paper}.
    \begin{assumption}\label{as:non_empty_primal}
    The problem \eqref{eq:gmp_obj} admits a feasible solution.
    \end{assumption}

\subsection{Slater's condition and compactness} 
\label{sec:gmp_properties}

To guarantee the existence of solutions to the Generalized Moment Problem, we make the following assumption:

\begin{assumption}[$S$-fullness]\label{as:x-full}
The $S$-fullness condition is said to be satisfied for problem \eqref{eq:gmp_obj} when there exists $b = \mult{h}{w}$, $w \in K^*$, such that $b_i>0$ on $S_i$ for $i=1, \ldots m$.
\end{assumption}
Note that $S$-fullness holds trivially in the case $\Omega$ describes a product of probability measures: order the matrix $h$ so that $(h_{i,j})_{1 \le i, j \le m} = \mathrm{Id}$ and set $w = (\overbrace{1,\ldots,1}^{m}, 0, 0 \ldots)$. Then, 
$b=(\overbrace{1,\ldots,1}^{m})$ and $b_i = 1 > 0$ for all $i=1,\ldots,m$.

Observe that the problem being $S$-full implies that $\Omega$ is bounded: 
\begin{lemma}\label{lem:compactness_p_star}
    Under \Cref{as:x-full}, the set $\Omega$ is weak$^*$ compact in the product topology. Furthermore, the value of $\pop^*$ is attained.
\end{lemma}
\begin{proof}
See \Cref{app:lp_proofs}.
\end{proof}

We consider an interior point criterion for strong duality between problems $\pop^*$ and $\dop^*$: 

\begin{condition}[Slater's Condition]\label{con:slaters}
    Let $X$, $Z$ be Banach spaces, and $Y = X^*$ and $V = Z^*$ their respective duals.
    Let $A : X \rightarrow Z$ be a continuous linear operator. Define $P$ as a 
    positive cone in $X$, and $Q$ a 
    positive cone in $Z$. Let $b \in Z$ and $c \in Y$. A problem of the form
    \begin{equation}\label{eq:anderson_IP}
        \begin{split}
            \textnormal{inf} \ \ &\langle x, c \rangle \\
            \textnormal{s.t.} \ &Ax - b \in Q, \\
            &x \in P,
        \end{split}
    \end{equation}
    is said to satisfy Slater's Condition when it has a finite infimum, and $\exists \ x_0 \in P$ such that $Ax_0 - b$ is in the interior of $Q$.
\end{condition}
    The importance of Slater's Condition lies in its establishment of strong duality between problem \eqref{eq:anderson_IP} and its dual, given by:
    \begin{equation}\label{eq:anderson_IP*}
        \begin{split}
            \textnormal{sup} \ \ &\langle b, w \rangle \\
            \textnormal{s.t.} \ & c -A^*w  \in P^*, \\
            &w \in Q^*.
        \end{split}
    \end{equation}

\begin{theorem}[{\cite[Theorem 10]{anderson1983review}}]
    If Slater's Condition is satisfied for \eqref{eq:anderson_IP}, then strong duality holds between the problems \eqref{eq:anderson_IP} and \eqref{eq:anderson_IP*}.
\end{theorem}

\begin{remark}
  In \cite[Theorem 10]{anderson1983review}, the hypothesis requires an element from the Mackey interior. Since we work in Banach spaces, which are locally convex, the Mackey topology coincides with the norm topology on convex sets \cite[Proposition 4, Chapter IV]{bourbaki2013topological}. Therefore, we only consider points in the interiors with respect to the norm induced by our Banach spaces.  
\end{remark}
\begin{proposition}\label{prop:slaters_d_star}
    Under Assumption \ref{as:x-full}, strong duality holds between $\dop^*$ and $\pop^*$.
\end{proposition}
\begin{proof}
See \Cref{app:lp_proofs}.
\end{proof}

While strong duality $\pop^* = \dop^*$ and attainment of $\pop^*$ hold under \Cref{as:x-full}, for the subsequent analysis of convergence rates, we require also that the dual optimum is attained:
\begin{assumption}\label{as:dual_optimizer}
The problem \eqref{eq:gmp_dual_obj} attains its supremum at $v^*$.
\end{assumption}
The satisfaction of this assumption depends on the specific structure of the GMP considered. It holds naturally for some GMPs such as in the polynomial optimization setting, and in the symmetric tensor decomposition setting (see \Cref{prop:attain} and \Cref{prop:sup_attainment_reals}). For the so-called $\cl A$-truncated $S$-moment problem, a condition of strict $S$-positivity is used to show the existence of a maximizer \cite[Proposition 3.6]{Nie2014}. In the optimal control setting, \Cref{as:dual_optimizer} is guaranteed by controllability conditions \cite[Equation (3.7)]{gaitsgory2012approximate}, \cite[Proposition 8(ii)]{gaitsgory2006linear}).

%

\section{Moment-SoS relaxations of the Generalized Moment Problem}\label{sec:moment-SOS}
In the following sections, the usage of $i$ refers to an arbitrary element from the sequence $\{1,\ldots,m\}$, where $m$ is the dimension of the generalized moment problem. \\
Let $\rng_i$ be the ring of polynomials in $n_i$ variables. We view $\rng_i$ as a subspace of $\mathcal{C}(S_i)$, the space of continuous functions on the compact set $S_i$. By the Stone-Weierstrass theorem, $\rng_i$ is a dense subspace of $\mathcal{C}(S_i)$ with respect to the sup-norm $\| \cdot \|_\infty$.\\
Any measure $\mu_i \in \mathcal{M}(S_i)$ defines a linear functional $L_{\mu_i}$ on $\rng_i$ via integration:
$$L_{\mu_i}(p) = \int_{S_i} p(x) \, d\mu_i(x) \quad \text{for all } p \in \rng_i.$$
The functional $L_{\mu_i}$ is an element of $\mathcal{D}_i \BYDEF \rng_i^*$, the algebraic dual of $\rng_i$. We equip $\mathcal{D}_i$ with the weak* topology, where a sequence of functionals converges if it converges pointwise for every polynomial in $\rng_i$.
The Riesz-Haviland theorem gives a necessary and sufficient condition for a linear functional $L \in \mathcal{D}_i$ to be representable by a positive measure on $S_i$. Specifically, such a representing measure $\mu_i \in \mathcal{M}_+(S_i)$ exists if and only if $L(p) \geq 0$ for all polynomials $p \in \rng_i$ that are non-negative on $S_i$. This allows us to work directly with linear functionals on polynomials instead of measures.

Hereafter, we assume that the compact sets $S_i$ are defined as
\begin{equation}\label{eq:def Si}
    S_i = \{x \in \RR^{n_i}, g_{i,1}(x) \ge 0, \ldots, g_{i,r_i}(x)\ge 0\}
\end{equation}
for some polynomial vectors $g_i = (g_{i,1}, \ldots, g_{i,r_i})\in \rng_i^{r_i}$, where $\rng_i^{r_i}$ is the $r_i$-fold product space.
We denote by $\pos(S_i)$ the convex cone of polynomials $p\in \rng_i$ such that $\forall x\in S_i$, $p(x) \ge 0$. 
Hereafter, we consider $S_i$ to be a set of the form $S_i = g_i^{-1}(\RR^{r_i}_+)$ with $g_i \in \rng_i^{r_i}$.
Let  $\Sigma^{2}_{i}= \{ q_1^2+\dots+q_s^2 \mid s \in \NN, \ q_k \in \rng_{i} \, \}\subset \rng_{i}$ be the convex cone of Sums of Squares (SoS).
The \emph{quadratic module} generated by the polynomials 
$g_i \in \rng_i^{r_i}$ is 
$$
\cl Q(g_i) \BYDEF \Sigma_{i}^2+ \sum_{j=1}^{r_i}\Sigma_{i}^2\cdot g_{i,j}.
$$
Hereafter, we will make the following assumption: 
\begin{assumption}[Archimedean quadratic module]\label{as:archimedean}
    For all $i=1, \ldots, m$, the quadratic module $\cl Q(g_i)$ contains a polynomial of the form $R_i^2 - (X_1^2 +\cdots+ X_{n_i}^2)$, where $R_i \in \RR_{+}$. 
\end{assumption}
This assumption is easy to satisfy when $S_i$ is bounded, that is, included in a ball of some radius $R_i$ centered at the origin, since we can add the sign constraint $R_i^2 - (X_1^2 +\cdots+ X_{n_i}^2)\ge 0$ to $g_i$ without changing $S_i$.

We define the dual cones:
    $\cl L(g_i)= \cl Q(g_i)^* = \{\lambda \in \cl D_i \mid \forall q \in \rng_i,\  \apply{\lambda}{g_{i,j} q^2} \ge 0, \ \apply{\lambda}{q^2} \ge 0, \ \forall j = 1,\ldots, r_i\}$
and $\cl L=\prod_{i=1}^m \cl L(g_i), 
\cl Q=\prod_{i=1}^m \cl Q(g_i).$ Define also the product spaces $\D \BYDEF \prod_{i=1}^m \D_i$ and $\rng \BYDEF \prod_{i=1}^m \rng_i$.

Under \Cref{as:archimedean}, 
it is readily verified by Putinar's Positivstellensatz that $\cl L(g_i) = \cl M_{+}(S_i)$.
Therefore, given a cost function $f\in \rng$ and $h_{i,j}\in \rng_i$ for $j \in J$, the primal relaxation of \eqref{eq:gmp_obj} translates as 
\begin{equation}\label{eq:relax_obj}
        \begin{split}
            {\pop}^* = \ \  \inf\ \ &\apply{\lambda}{f} \\
            \ \ \ \ \textnormal{s.t.} \ \ &t - \apply{\lambda}{h} \in K, \\
            &\lambda \in \cl L,
        \end{split}
\end{equation}
whose feasible set is denoted
\begin{equation}\label{eq:relax_con}
    \O = \{ \lambda \in \cl L  \mid \ t-\apply{\lambda}{h} \in K \}.
\end{equation}
We consider also the dual SoS relaxation
\begin{equation}
        \begin{split}
            {\dop}^* = \ \  \sup\ \ &\mult{v}{t} \\
            \ \ \ \ \textnormal{s.t.} \ \ &f - \mult{h}{v} \in \cl Q  \\
            &v \in - K^*
        \end{split}
\end{equation}
whose feasible set is 
\begin{equation}\label{eq:relax_dual_con}
E = \left\{v \in -K^* : f - \mult{h}{v} \in \cl Q \right\}.
\end{equation}
It is readily verified under \Cref{as:archimedean} that $\pos = \overline{\cl Q}$, where the closure of $\cl Q$ is taken in the weak$^*$ topology 
($p_n \xrightarrow{w} p$ in $\rng$ iff $\forall \lambda \in \D$, $\apply{\lambda}{p_n} \to \apply{\lambda}{p}$, which is equivalent to point-wise convergence). We deduce from the previous section the following result:
\begin{lemma}\label{lem:compactness_relaxation}
    Under \Cref{as:x-full} and \Cref{as:archimedean}, $\O$ is weak$^*$ compact and $\pop^*=\dop^*$. Furthermore, the value of ${\pop}^*$ is attained.
\end{lemma}

\subsection{Convergence of the optima}\label{sec:convergence_results}
We now introduce the Moment-SoS hierarchies, which consist of a sequence of convex optimization problems indexed by degree $\ell \in \NN$.  
For a vector $f=[f_1,\ldots, f_m] \in \rng$, let $\deg(f) = \max_{i=1,\ldots,m} \deg(f_i)$.
For $\ell \in \NN$,  let  $\Sigma^{2}_{i,\ell}= \{ \, q_1^2+\dots+q_s^2 \mid s \in \NN, \ q_k \in \rng_{i,\lfloor \ell/2 \rfloor} \, \}\subset \rng_{i,\ell}$ be the convex finite-dimensional cone of Sums of Squares (SoS) of degree $\le \ell$.
The \emph{truncated quadratic module} generated by the polynomials 
$g_i \in \rng_i^{r_i}$ in degree $\le \ell$ is 
$$
\cl Q_{\ell}(g_i) \BYDEF \Sigma_{i,\ell}^2+ \sum_{j=1}^{r_i}\Sigma_{i,\ell-\deg(g_{i,j})}^2\cdot g_{i,j}.
$$
Let $\cl Q_\ell = \prod_{i=1}^m \cl Q_\ell(g_i)$, 
 $\cl L_\ell(g_i)= \cl Q_\ell(g_i)^*= \{\lambda \in \cl D_i \mid \forall q \in \cl Q_\ell(g_i), \apply{\lambda}{q}\ge 0\}$, 
and $\cl L_\ell = \prod_{i=1}^m \cl L_\ell(g_i)$. 

Henceforth, we specify $f$ and $h_{j}$ to be elements of $\rng$. 
Define 
$$
J_{\ell} = \{ j \in J \mid \textup{ for all } i \in \{1, \ldots, m\}, \ \deg(h_{i,j}) \le \ell \}
$$
and let $h^\ell \BYDEF (h_{i,j})_{1 \le i \le m, j \in J_{\ell}}$, $t^\ell \BYDEF (t_{j})_{j \in J_{\ell}}$ and $K^\ell \BYDEF \{ u^\ell \in \mathbb{R}^{J_\ell} \mid \exists \ u \in K$ such that $u_j = u^\ell_j \text{ for all } j \in J_\ell \}$. 
We then define the corresponding dual cone $(K^\ell)^* \subseteq \mathbb{R}_0^{J_\ell}$ as
$$
(K^\ell)^* = \{ v^\ell \in \mathbb{R}_0^{J_\ell} \mid \forall \  u^\ell \in K^\ell, \sum_{j \in J_\ell} u^\ell_j v^\ell_j \ge 0 \}.
$$
The cone $K^\ell$ is the projection of the original cone $K \subseteq \mathbb{R}_+^J$ onto the coordinates indexed by $J_\ell$. 
\noindent We consider the following hierarchy of truncated primal relaxations
\begin{equation}
        \begin{split}
            {\pop}_{\ell}^* = \ \  \inf\ \ &\apply{\lambda}{f} \\
            \ \ \ \ \textnormal{s.t.} \ \ &t^\ell - \apply{\lambda}{h^\ell} \in K^\ell, \\
            &\lambda \in \cl L_{\ell},
        \end{split}
\end{equation}
along with the corresponding dual $\ell$-degree SoS relaxations:
\begin{equation}\label{eq:relax_dual_obj}
        \begin{split}
            {\dop}_{\ell}^* = \ \  \sup\ \ &\mult{v}{t^\ell} \\
            \ \ \ \ \textnormal{s.t.} \ \ &f - \mult{h^\ell}{v} \in \cl Q_{\ell}, \\
            &v \in - (K^\ell)^*.
        \end{split}
\end{equation}

It is readily verified that weak duality holds: $\dop^*_{\ell} \le \pop_\ell^*$ and that $\dop^*_{\ell}\le \dop_{\ell+1}^* \le \dop^*$, $\pop_{\ell}^*\le \pop_{\ell+1}^* \le \pop^*$.

To analyze the convergence of the hierarchy, we will use the following Positivstellens\"atze, which hold under \Cref{as:archimedean}:
\begin{theorem}[{Effective Positivstellens\"atze}]\label{thm:putinar_effective}
For $n \ge 2$, $g \subset \rng$ finite such that $\cl Q(g)$ satisfies Assumption \ref{as:archimedean}, $S=\{x \in \RR^{n}\mid \forall p\in g, \ p(x)\ge 0\}$ and $p$ strictly positive on $S$, one has that $p \in \cl Q_\ell(g)$ for

$$
\ell \ge \max \{ \gamma \left(\frac{p_{\max}}{p_{\min}}\right)^{\frac 1 \theta}, \ell_0\}.
$$
where $\theta>0$, and $\gamma$, $\ell_0$ are constants depending on $n$, $g$ and $\deg(p)$, and $p_{\min} \BYDEF \min_{x \in S} p(x)$, $p_{\max} \BYDEF \max_{x \in S} p(x)$. We have 
\begin{itemize}
    \item $\theta = \frac 1 {2.5n\lo}$, $\gamma=\gamma'(n,g) \deg(p)^{-3.5n\lo}$ and $\ell_0=0$, where $\gamma'(n,g)$ is a constant depending on $n$ and $g$, $\lo$ is the exponent in the  Łojasiewicz  inequality between the semi-algebraic distance defined from the polynomials $g$ and the Euclidean distance to $S$ (see \cite[Theorem 1.7]{baldi2023effective}), 
    \item $\theta=1$,  $\gamma$ depends polynomially on $deg(p)$ for $n$ fixed and $\ell_0 = \pi \, n\, \sqrt{2n} \deg(p)$,  when $S$ is the box $S=[-\rho, \rho]^{n}$ and  $g = \{ \rho^2-x_i^2,   i= 1, \ldots, n\}$ (see \cite[Theorem 11]{BaldiSlot2024}),
    \item $\theta = 2$,  $\gamma$ depends polynomially on $deg(p)$ for $n$ fixed and $\ell_0 = \pi \, n\, \sqrt{2n} \deg(p)$,  when $S$ is the box $S=[-1, 1]^{n}$ and $g = \{ \prod_{i=1}^{n} (1-x_i^2)^{e_i}, \textup{ for }  e =[e_1, \ldots, e_n] \in \{0,1\}^{n}\} $ (see \cite[Corollary 3]{laurent2023effective}),
    \item $\theta = 2$, $\gamma$ depends polynomially on $deg(p)$ for $n$ fixed and $\ell_0= 2 n \deg(p)^{\frac 3 2}$, when $S= \{x \in \RR^n\mid 1-\norm{x}^2 \ge 0\}$ is the unit ball 
(see \cite[Theorem 3]{Slot2022}).
\end{itemize}
\end{theorem}

The main result of the section is as follows: 
\begin{theorem}\label{thm:gmp_convergence}
    Let \Cref{as:x-full}, \Cref{as:dual_optimizer} and \Cref{as:archimedean} be satisfied. Let $\dop^*$ be as in \eqref{eq:gmp_dual_obj}  and $\dop_\ell^*$ be as in \eqref{eq:relax_dual_obj}. For 
    \begin{equation}\label{eq:ell_bound}
    \ell \ge \max\left(\max_{i=1, \ldots,m} \ell_{i,0}, \max_{ j\in \textup{supp}(v^*) \cup \textup{supp}(w)}\deg(h_j)\right),
    \end{equation}
    we have
     $$0 \le \dop^* - \dop^*_\ell \le 
    \kappa \, \ell^{-\theta}$$
where

\begin{equation}\label{eq:kappa theta}
\begin{array}{rl}
    \kappa & = 
    \max_{i=1,\ldots,m}\frac{\mult{t}{w}}{(\mult{h_i}{w})_{\min}} {\gamma_i}^{\theta_i}(2 f_{i,\max} + \norm{v^*}_1 h^*_{i,\max} )\\
    \ \theta &= \min_{i=1,\ldots,m} \theta_i
\end{array}
\end{equation}

    with 
    \begin{itemize} 
        \item $f_{i,\max} \BYDEF \max_{x \in S_i}f_i(x)$
        \item $h^*_{i,\max} \BYDEF \max_{j \in \textup{supp}(v^*)} \max_{x \in S_i} |h_{i,j}(x)|$
        \item $\gamma_i$, $\theta_{i}$ as in \Cref{thm:putinar_effective} for $S_i$, 
        \item $ w $ as in Assumption \ref{as:x-full},
        \item $v^*$ the vector achieving the supremum in $\dop^*$.
    \end{itemize}
    
\end{theorem}

\begin{proof}
By \Cref{as:dual_optimizer}, the dual optimizer $v^*$ of \eqref{eq:gmp_dual_obj} exists.
For any $\e > 0$, define the perturbed vector $v_\e$ by
\begin{equation}\label{eq:sdp_perturbation}
v_\e = v^* - \e w,
\end{equation}
where $w$ is as in \Cref{as:x-full}. 
Accompanying $v_\e$, we have the strictly positive polynomial sequence $q = (q_i)_{i=1}^m \in \cl C$ defined as follows:
$$
q \BYDEF f - \mult{h}{v_\e} = f - \mult{h}{v^*}  + \e \mult{h}{w} > 0,
$$
by the feasibility of $v^*$ in \eqref{eq:gmp_dual_obj}, and since $\mult{h}{w} > 0$ by \Cref{as:x-full}.
Note that the premise for \Cref{thm:putinar_effective} holds for this $q \in \pos$. We have  
$$
\ell \ge\max_{i=1\ldots,m} \max  \left(\gamma_i \left(\frac{q_{i,\max}}{q_{i,\min}}\right)^{\frac{1}{\theta_i}}, \ell_{i,0}\right) \implies q \in \cl Q_{\ell}.
$$
For each $i=1,\ldots,m$, by the feasibility of $v^*$ in \eqref{eq:gmp_dual_obj}, 
\begin{equation}
q_{i,\min} \ge (f_i - \mult{h_i}{v^*})_{\min} + (\e \mult{h_i}{w})_{\min} \ge \e (\mult{h_i}{w})_{\min} > 0, \label{eq:tmp_psatz_l_1}
\end{equation}
where $(\mult{h_i}{w})_{\min} > 0$ by \Cref{as:x-full}. 
Letting 
$$
{h}^*_{i,\max} = \max_{j \in \textup{supp}(v^*)} \max_{x \in S_i} |h_{i, j}(x)|,
$$ we have
\begin{align}
q_{i,\max} & = (f_i-\mult{h_i}{v_\e})_{\max} \le f_{i,\max}+ (-\mult{h_i}{v^*})_{\max} + \e (\mult{h_i}{w})_{\max}, \notag \\ 
& \le f_{i,\max}  + \|v^*\|_1 {h}^*_{i,\max} + \e (\mult{h_i}{w})_{\max}, \notag \\
& \le 2 f_{i,\max} + \norm{v^*}_1 {h}^*_{i,\max}, \label{eq:tmp_psatz_l_2}
\end{align}
where the final inequality follows by choosing $\e \le  f_{i,\max} / ( \mult{h_i}{w} )_{\max} $. 
By \eqref{eq:tmp_psatz_l_1} and \eqref{eq:tmp_psatz_l_2}, $q \in \mathcal Q_\ell$ for 
\begin{equation}\label{eq:tmp_psatz_l}
\ell \ge \max_{i=1\ldots,m} \max \left(\gamma_i \left(\frac{2 f_{i,\max} + \norm{v^*}_1 h^*_{i,\max}}{\e (\mult{h_i}{w})_{\min}}\right)^{\frac{1}{\theta_i}}, \ell_{i,0}\right).
\end{equation}
By the optimality of $v^*$, we have that $\dop^* - \mult{t}{v^*} = 0$. By \eqref{eq:sdp_perturbation}, we have that
$ \mult{t}{v^*} - \e \, \mult{t}{w}- \mult{t}{v_\e} = 0.
$ Therefore,
$$
0 \le \dop^* - \dop^*_{\ell} =  \dop^* - (\mult{t}{v^*} - \e \, \mult{t}{w}- \mult{t}{v_\e}) - \dop^*_\ell 
= \mult{t}{v_\e} - \dop^*_\ell + \e \,\mult{t}{w},
$$
By \eqref{eq:ell_bound}, $\textup{supp}(v_\e) \subseteq J_\ell$,
implying that 
$$
f - \mult{h}{v_\e} = f - \mult{h^\ell}{v_\e} \in \cl Q_\ell
$$
and
$$
\sum_{j \in J_\ell} v_{\e, j} u^\ell_j \le 0
$$
for all $u^\ell \in K^{\ell}$, since $v_\e \in -K^*$.
That is, $v_\e$ is feasible in $\dop_\ell^*$, and $\mult{t}{v_\e} - \dop^*_\ell \le 0$. Thus
$$
0 \le \dop^* - \dop^*_\ell \le \e(\mult{t}{w}). 
$$
Note that $\mult{t}{w} > 0$ (see \eqref{eq:t_dot_w} in the proof of \Cref{lem:compactness_p_star}). Rearranging \eqref{eq:tmp_psatz_l} for $\e$ and substituting into the above gives the result for $\ell \ge \max\big(\max_{i=1, \ldots,m} \ell_{i,0}$, $\max_{ j\in \textup{supp}(v^*) \cup \textup{supp}(w)}\deg(h_j)\big)$: 
$$
0 \le \dop^* - \dop^*_\ell \le  \max_{i=1,\ldots,m}\left(\frac{\mult{t}{w}}{(\mult{h_i}{w})_{\min}} \left(\frac{\gamma_i}{\ell} \right)^{\theta_i}(2 f_{i,\max} + \norm{v^*}_1 h_{i,\max} )\right) \le \kappa \, \ell^{-\theta}.
$$
with $\kappa$ and $\theta$ as in \eqref{eq:kappa theta}.

\end{proof}

\begin{corollary}\label{cor:gmp_convergence}
    Under Assumption \ref{as:x-full}, \Cref{as:dual_optimizer} and Assumption \ref{as:archimedean},
    $$
    0 \le \pop^* - \pop_\ell^* \le \dop^* -\dop_{\ell}^* \le \kappa\, \ell ^{-\theta} 
    $$
    with $\kappa, \theta$ as in \Cref{thm:gmp_convergence}.
\end{corollary}
\begin{proof}
    Follows from strong duality in the generalized moment problem (\Cref{prop:slaters_d_star}), weak duality in the Moment-SoS hierarchy and \Cref{thm:gmp_convergence}.
\end{proof}

\subsection{Convergence of the moment feasibility sets}\label{sec:hausdorff_dist}

In this section, we establish the convergence between the GMP feasibility set \eqref{eq:relax_con} and the feasibility set of the relaxations.

Recall that $\rng_{i,k}$ is the space of polynomials in $\rng_i$ of degree at most $k$.
For $k \in \NN$ and $\lambda \in \drng_i$, let $\lambda^{(k)}\in \drng_i$ such that $\lambda^{(k)}_{|\rng_{i,k}} = \lambda_{|\rng_{i,k}}$ and $\apply{\lambda^{(k)}}{\bm x_i^{\a}}=0$ if $|\a|>k$. 
Here, $\lambda^{(k)}$ is the k-truncation of the linear functional $\lambda$, meaning its moments agree with $\lambda$ up to degree $k$ and are zero for degrees greater than $k$.

Let $\drng_i^{(k)}=\{\lambda^{(k)} \mid \lambda \in \drng_i\}$ and 
$\drng^{(k)} = \prod_{i=1}^{m} \drng_i^{(k)}$.
To study the distance between spaces of linear functionals, we equip $\drng_i$ with the weighted norm $\norm{ \cdot }_i$ defined such that for $\bm x_i = (\bm x_{i,1},\ldots, \bm x_{i,n_i}$),
$$\norm{\lambda}_i \BYDEF \sup_{\a \in \mathbb{N}^{n_i}} \frac{|\apply{\lambda}{\bm x_i^\a}|}{\a !}.$$
This norm was applied in \cite{lasserre2013K} to ensure continuity of linear functionals over closed (not necessarily compact) semi-algebraic sets.
We denote
$\drng^{\bnd}_i=\{\lambda \in \drng_i  \mid \norm{\lambda}_i < \infty\}$ and $\drng^{\bnd} = \prod_{i=1}^{m} \drng^{\bnd}_i$ equipped with the product norm, which we denote with $\norm{\cdot}$. The following lemma establishes that the space of measures on $S_i$ can be embedded into the normed space $\drng^{\bnd}_i$.

\begin{lemma}\label{lem:boundeness_measure}
Every finite signed Borel measure $\mu$ on the compact set $S_i$ has a finite weighted norm, namely $\norm{\mu}_i \le \sup_{\a} \frac{R_i^{|\a|}} {\a!}\mu(S_i)$ for $R_i>0$ such that $S_i \subset B_{R_i}(\RR^{n_i})$. That is,
$$
    \cl M(S_i) \subset \cl D^{\bnd}_i.
$$
\end{lemma}
\begin{proof}
By the Jordan decomposition, any signed measure $\mu \in \cl M(S_i)$ can be written as $\mu = \mu_+ - \mu_-$ where $\mu_+, \ \mu_- \in \Mp(S_i)$. By the triangle inequality for the norm, if we can show that $\cl M_+(S_i) \subset \cl D^{\bnd}_i$, the result will follow.
Let $\mu \in \Mp(S_i)$. Since $S_i$ is compact, $\forall x \in S_i$, $|x^\a| \le R^{|\a|}$ for some $R > 0$, implying $|\mu(\bm x_i^\a)| \le R^{|\a|} \mu(S_i)$, where $|\a| \BYDEF \sum_{i=1}^{n_i} \a_i$.  Then $\| \mu \|_i = \sup_{\a} | \mu(\bm x_i^\a) | / \a ! \le \sup_{\a} R^{|\a|}\mu(S_i) / \a ! < \infty$, so $\mu \in \cl D^{\bnd}_i$.
\end{proof}

Let the norm of the dual space $(\cl D_i^{\bnd})^*$ be given by the standard operator norm $$\norm{f}^*_i \BYDEF \sup \{ |f(\l)| / \norm{\l}_i : \l \in \drng_i^\bnd, \l \neq 0 \}.$$ 
We present now  a characterization of the dual space $(\cl D_i^{\bnd})^*$ that we will need hereafter. Let $A$ be the space of real-analytic functions $f = \sum_{\a \in \mathbb{N}^{n_i}} f_\a \bm x^\a$ equipped with the norm $\norm{f}_A \BYDEF \sum_{\a \in \mathbb{N}^{n_i}} \a! | f_\a |$. 
    Define the mapping $\ph : (\cl D_i^{\bnd})^* \rightarrow A$ by
    $$
    \ph(f)(\bm x) \BYDEF \sum_{\a \in \mathbb{N}^{n_i}} f(\d_\a) \bm x^\a.
    $$

The following classical lemma establishes that the mapping $\ph$ is an isometric isomorphism:
\begin{lemma}\label{lem:dual_analytic}
    The space $(\drng^{\bnd}_i)^*$ equipped with the norm $\norm{\cdot}^*_i$ is isometrically isomorphic to the space of real-analytic functions $A$ with norm $\norm{f}_A \BYDEF \sum_{\a \in \mathbb{N}^{n_i}} \a! | f_\a |$. 
\end{lemma}
For the sake of completeness, we provide its proof in \Cref{app:lem}.

The isometric isomorphism $\ph: (\mathcal{D}^\bnd_i)^* \to A$ established above allows us to identify elements in the two spaces. Consequently, we will use $f$ interchangeably to denote both a functional in $(\mathcal{D}^\bnd_i)^*$ and its corresponding analytic function in $A$, noting that the isometry implies their norms are equal, i.e., $\|f\|^* = \|f\|_A$.

We define the convex cones
$\cl L_{\ell}^{\bnd}(g_i)= \cl L_{\ell}(g_i) \cap \drng_i^{\bnd}$, 
$\cl L_\ell^{\bnd} = \prod_{i=1}^m \cl L^{\bnd}_\ell(g_i)$
 and the corresponding relaxation: 
$$
\begin{array}{rlcrl}
{\pop}_{\ell}^{\bnd} = \ \  \inf\ \ &\apply{\lambda}{f} \\
 \ \ \ \ \textnormal{s.t.} \ \ &t^\ell - \apply{\lambda}{h^\ell} \in K^\ell \\
            &\lambda \in \cl L^{\bnd}_{\ell}. 
\end{array}
$$
The feasible set of the primal problem is denoted $L^{\bnd}_\ell= \{\lambda \in \cl L^{\bnd}_{\ell}  \mid \ t^\ell-\apply{\lambda}{h^\ell} \in K^\ell \}$.
\begin{lemma}\label{lem:p_flat_p}
For $\ell\ge \deg(f)$
, $\pop_{\ell}^{\bnd}= \pop_{\ell}^*$. 
\end{lemma}
\begin{proof}
    Clearly $ \pop_{\ell}^* \le \pop_{\ell}^{\bnd}$.
    Let $(\lambda_k)_{k\in\mathbb{N}}$ be a minimizing sequence in $\pop_{\ell}^*$ converging to $\lambda^*$. As $\ell\ge \deg(f)$
    , we have that $\apply{\lambda_k^{(\ell)}}{f} = \apply{\lambda_k}{f}$ and, by definition of $h^\ell$, $t - \apply{\lambda_k^{(\ell)}}{h^\ell} = t - \apply{\lambda_k}{h^\ell} $ for all $k$. That is, the sequence $\lambda^{(\ell)}_k$ is feasible in $\pop_{\ell}^{\bnd}$, belongs to $\cl L_{\ell}^{\bnd}$ and achieves the same infimum  $\pop_{\ell}^*$. Thus $\pop^{\bnd}_{\ell}\le \pop_{\ell}^{*}\le \pop^{\bnd}_{\ell}$, as required.  
\end{proof}
\brev
\begin{remark}\label{rem:canonical_sol}
    For any optimizer $\lambda^*_\ell$ of $\pop^*_\ell$, its $\ell$-truncation $\lambda^{* (\ell)}_\ell$ is also optimal in $\pop^*_\ell$. We thus refer to such optimizers satisfying $\lambda^*_\ell = \lambda^{* (\ell)}_\ell$ as \emph{canonical}. 
\end{remark}
\erev

Denoting the distance associated to the product norm $\norm{\cdot}$ as $\rho$, we can define the ``distance" $\rho(\mu,\Gamma)$ between $\mu\in \cl D^{\bnd}$ and $\Gamma\subset \cl D^{\bnd}$
and the Hausdorff distance $\rho_H(\Gamma_1,\Gamma_2)$ between $\Gamma_1\subset \cl D^{\bnd}$ and $\Gamma_2\subset \cl D^{\bnd}$ as follows:

\begin{equation}\label{eq:measures-distance-definitions}
\rho(\mu,\Gamma)\BYDEF\inf_{\mu'\in \Gamma}\rho(\mu,\mu'),\quad
\rho_H(\Gamma_1,\Gamma_2)\BYDEF\max\{\sup_{\mu\in \Gamma_1}\rho(\mu,\Gamma_2),\sup_{\mu\in \Gamma_2}\rho(\mu,\Gamma_1)\}.
\end{equation}
\begin{lemma}\label{lem:hausdorff_conv}
    Assume the value of $\pop_\ell^*$ is attained for any $f\in \rng$. Let $k \in \mathbb{N}$. For any $c \ge 0$ and any $\ell \ge k$, the following equivalence holds:
    $$
    ( \forall f \in \rng_k \text{ with } \| f \|_A = 1, \, 0 \le \pop^* - \pop_\ell^* \le c ) \iff \rho_H(\O^{(k)}, L_\ell^{(k)}) \le c,
    $$
    where $\O^{(k)} \BYDEF \{\lambda^{(k)} \mid \lambda \in \O\}$ and $L_\ell^{(k)} \BYDEF \{\lambda^{(k)} \mid \lambda \in L^{\bnd}_\ell\}$ are the sets of moment sequences truncated to degree $k$.
\end{lemma}
\begin{proof}
    We first prove the backward direction $(\Leftarrow)$. Assume $\rho_H(\O^{(k)}, L_\ell^{(k)}) \le c$ for some $\ell \ge k$. Let $f \in \rng_k$ with $\|f\|_A = 1$.
    Since $\O \subseteq L^{\bnd}_\ell$, it follows that $\O^{(k)} \subseteq L_\ell^{(k)}$ and that $\pop_\ell^* \le \pop^*$.
    Let $\lambda_\ell^*$ be an optimizer for $\pop_\ell^*$, which exists by assumption. Its truncation $\lambda_\ell^{*(k)}$ lies in $L_\ell^{(k)}$. The fact $\rho_H(\O^{(k)}, L_\ell^{(k)}) \le c$ implies the existence of a $\lambda \in \O$ such that its truncation $\lambda^{(k)}$ satisfies $\|\lambda_\ell^{*(k)} - \lambda^{(k)}\| \le c$.
    Since $\deg(f) = k$, we have $\apply{f}{\mu} = \apply{f}{\mu^{(k)}}$ for any measure $\mu$. We can write
    \begin{align*}
    \pop_\ell^* = \apply{f}{\lambda_\ell^*} = \apply{f}{\lambda_\ell^{*(k)}} &= \apply{f}{\lambda^{(k)}} + \apply{f}{\lambda_\ell^{*(k)} - \lambda^{(k)}} \\
    &= \apply{f}{\lambda} + \apply{f}{\lambda_\ell^{*(k)} - \lambda^{(k)}}.
    \end{align*}
    By definition, $\apply{f}{\lambda} \ge \pop^*$. The second term is bounded using the dual norm. Let $f$ also denote the linear functional in $(\cl D^{\bnd})^*$ corresponding to the polynomial $f \in A$. By the isometry in \Cref{lem:dual_analytic}, $\|f\|^* = \|f\|_A$, implying
    $$
    \apply{f}{\lambda_\ell^{*(k)} - \lambda^{(k)}} \ge - \left|\apply{f}{\lambda_\ell^{*(k)} - \lambda^{(k)}}\right| \ge -\|f\|^* \cdot \|\lambda_\ell^{*(k)} - \lambda^{(k)}\| \ge -c.
    $$
    Combining the equations gives $\pop_\ell^* \ge \pop^* - c$, which implies $\pop^* - \pop_\ell^* \le c$, as required.\\
    We now prove the forward direction $(\Rightarrow)$ by contradiction. Assume that for any $f \in \rng_k$ with $\|f\|_A=1$, we have $0 \le \pop^* - \pop_\ell^* \le c$ for any $\ell \ge k$. Suppose, for the sake of contradiction, that $\rho_H(\O^{(k)}, L_\ell^{(k)}) > c$.
    This implies that there exists a $\lambda'_0 \in L_\ell^\bnd$ whose truncation $\lambda_0'^{(k)}$ is not in the set $\O^{(k)} + B_c(\mathcal{D}^{(k)})$, where $\mathcal{D}^{(k)}$ is the space of moment sequences up to degree $k$, and $B_c(\mathcal{D}^{(k)})$ the closed ball of radius $c$ in $\mathcal{D}^{(k)}$. The sets $\{\lambda_0'^{(k)} \}$ and $\O^{(k)} + B_c(\mathcal{D}^{(k)})$ are disjoint, compact and convex. Thus, they can be strictly separated by the separation theorem  (see, e.g., \cite[p. 59]{rudin1991functional}).
     There exists an element $q$ from the dual space of $\mathcal{D}^{(k)}$, such that
    $$
    \min_{\l \in \O^{(k)}+B_c(\cl D^{(k)})} \apply{\l}{q} - \apply{\lambda_0'^{(k)}}{q} = \beta 
    > 0.
    $$
    As a dual element to the space of truncated moment sequences, $q$ can be identified with a polynomial in $\rng_k$.
    Note that
    $$
    \min_{\l \in \O^{(k)}+B_c(\cl D^{(k)})} \apply{\l}{q} = \min_{\l' \in \O^{(k)}} \apply{\l'}{q} + \min_{\l'' \in  B_c(\cl D^{(k)})} \apply{\l''}{q}.
    $$
    $$
    = \min_{\l' \in \O^{(k)}} \apply{\l'}{q} - c \max_{\| \l'' \| = 1} \apply{\l''}{q}.
    $$
    $$
    = \min_{\l' \in \O^{(k)}} \apply{\l'}{q} - c \|q\|^*.
    $$
    Considering the problems $\pop^*$ and $\pop_\ell^*$ with objective $f = q / \| q \|_A$. 
    We obtain
    $$
    \pop^* - c> \pop_\ell^*,
    $$
    that is 
    $$
    \pop^* - \pop_\ell^* > c ,
    $$
    a contradiction.
\end{proof}

\begin{theorem}\label{thm:haus_dis}
    Under \Cref{as:x-full}, \Cref{as:dual_optimizer} and \Cref{as:archimedean}, for any fixed degree  $k \in \mathbb{N}$, with $\ell \ge k$ and $\ell$ satisfying \eqref{eq:ell_bound},
    $$\rho_H(\O^{(k)}, L^{(k)}_\ell) \le \kappa' \ell^{-\theta}$$
    where 
    $$
    \kappa' =  \max_{i=1,\ldots,m}\left( \frac{ \mult{t}{w}}{(\mult{h_i}{w})_{\min}} \gamma_i^{\theta_i}(2C_{S_i} + \norm{v^*}_1 h_{i,\max}  )\,  \right),
    $$
    and $C_{S_i} = \sup_{\alpha \in \mathbb{N}^{n_i}}  (\max_{x \in S_i} |x^\alpha|) / \alpha!$ is a constant depending on the set $S_i$. The other constants $\theta, \gamma_i, \lo_i, w, v^*, h_{i,\max}$ are as in \Cref{thm:gmp_convergence}.
\end{theorem}
\begin{proof}
    Take an objective function $f \in \rng_k$ such that $\| f \|_A = 1$. To apply \Cref{thm:gmp_convergence}, we need a uniform bound on $f_{i,\max}$. By triangle inequality and re-arranging terms, we have:
    \begin{align*}
    f_{i,\max} \le \|f_i\|_{\infty} &\le \sum_{\alpha} |f_{i,\alpha}| (\max_{x \in S_i} |x^\alpha|) \\
    &= \sum_{\alpha} (\alpha! |f_{i,\alpha}|) \frac{\max_{x \in S_i} |x^\alpha|}{\alpha!} \\
    &\le \left( \sum_{\alpha} \alpha! |f_{i,\alpha}| \right) \sup_{\beta \in \mathbb{N}^{n_i}} \left( \frac{\max_{x \in S_i} |x^\beta|}{\beta!} \right) = C_{S_i} \|f_i\|_A.
    \end{align*}
    Since the total norm is $\|f\|_A = \sum_{j=1}^m \|f_j\|_A = 1$, the individual norms $\|f_i\|_A$ are at most 1. This gives the uniform bound $f_{i,\max} \le C_{S_i}$. Substituting this bound into the result of \Cref{thm:gmp_convergence}, we obtain
    $$
    0 \le \pop^* - \pop_\ell^* \le \kappa' \ell^{-\theta}
    $$
    which, by \Cref{lem:hausdorff_conv}, implies the result
\end{proof}

The above result demonstrates that the cone of pseudo-moments $L_\ell^{(k)}$ converges to the cone of truncated moments of measures $\O^{(k)}$ with the same rate of convergence, $\theta$, as the optima, as $\ell \rightarrow \infty$ with $k$ is fixed.
That is, the pseudo-moments up to order $k$ of an optimum solution of the GMP $\pop_\ell^*$ are close to the moments of a measure for large enough $\ell$.

\brev{} We lastly establish the convergence of the optimizers of $\pop^*_\ell$ to the optimizer of $\pop^*$. Under the assumption that $\pop^*$ has a unique minimizer $\mu^*$, we show that sequences of canonical optimizers $(\lambda_\ell^*)$ (see \Cref{rem:canonical_sol}) converges to $\mu^*$ in the weak$^*$ topology. We require the following boundedness lemma.

\begin{lemma} 
\label{lem:uniform_bdd}
    Let Assumptions \ref{as:x-full} and \ref{as:archimedean} be satisfied. Let $(\lambda_\ell^*)$ be a sequence of canonical 
    optimizers for the truncated problems $(\pop_\ell^*)$. The sequence is uniformly bounded. That is, there exists a constant $\rho > 0$ such that $\|\lambda_\ell^*\| \le \rho$ for all $\ell$. 
\end{lemma}
\begin{proof}
    Let us first prove that $\O \subset B_{\rho}(\cl D^{\bnd})$ for some $\rho>0$.
    Let $\mu \in \O$. 
    By \Cref{lem:boundeness_measure} and \eqref{eq:kappa} in the proof of \Cref{lem:compactness_p_star}, and with $R$ such that $S_i \subset B_{R}(\RR^{n_i})$, we have 
    $$
    \norm{\mu} \le \sup \frac{R^{|\a|}}{\a!} \sum_i \apply{\mu_i}{1} 
    = \sup \frac{R^{|\a|}}{\a!} \frac{t\cdot\omega}{b_{\min}} \BYDEF \rho'.
    $$
    Consider now its $k$-truncation, $\mu^{(k)}$. By definition, 
    \begin{equation}\label{eq:uniform_bdd_1}
    \|\mu^{(k)}\| = \sup_{\a} \frac{|\apply{\mu^{(k)}}{\bm x^\a}|}{\a!} \le \sup_{\a} \frac{|\apply{\mu}{\bm x^\a}|}{\a!} = \|\mu\| \le \rho'.
    \end{equation}
    Thus, for any $k \in \mathbb{N}$, the set $\O^{(k)}$ is also contained in the same ball $B_{\rho'}(\cl D^{\bnd})$.
    
    Let $\bar{\ell}$ be the minimum $\ell$ such that \Cref{thm:haus_dis} holds. By \Cref{thm:haus_dis} with $k=\ell \ge \bar{\ell}$, we have
    $$
    \rho_H(\O^{(\ell)}, L^{(\ell)}_\ell) \le \kappa' \ell^{-\theta}.
    $$
    By definition of the Hausdorff distance, 
    $$
    \min_{\mu \in \O^{(\ell)}} \|\lambda_\ell^{*} - \mu\| \le \kappa' \ell^{-\theta}.
    $$
    By the weak$^*$ compactness of $\O$ and the continuity of its truncation to $\O^{(k)}$, there exists an element $\mu_\ell \in \O^{(\ell)}$ such that $\|\lambda_\ell^{*} - \mu_\ell\| \le \kappa' \ell^{-\theta}$.
    By \eqref{eq:uniform_bdd_1} and triangle inequality, we have
    $$
    \|\lambda_\ell^{*}\| \le \|\lambda_\ell^{*} - \mu_\ell\| + \|\mu_\ell\| \le \kappa' \ell^{-\theta} + \rho' \le \bar{\ell}^{-\theta} + \rho' \BYDEF \rho''.
    $$
    since the term $\ell^{-\theta}$ is decreasing with $\ell > 0$. 
    Thus, the entire sequence of optimizers $(\lambda_\ell^*)_{\ell}$ is contained in $B_{\rho}(\cl D^{\bnd})$ where $\rho = \max \{ \norm{\lambda^*_1}, \norm{\lambda^*_2}, \dots, \norm{\lambda^*_{\bar{\ell}-1}}, \rho'' \}$.
\end{proof}
We can now prove the optimizer convergence theorem.
\begin{theorem}\label{thm:unique opt}
    Let Assumptions \ref{as:x-full}, \ref{as:dual_optimizer}, and \ref{as:archimedean} be satisfied. Assume that problem $\pop^*$ has a unique minimizer $\mu^* \in \O$. Let $(\lambda^*_\ell)$ be a sequence of canonical 
    optimizers for the truncated problems $(\pop_\ell^*)$. Then, the sequence $(\lambda_\ell^*)$ converges to $\mu^*$ in the weak* topology.   
\end{theorem}
\begin{proof}
The proof is in three parts.

\noindent\textbf{1. Existence of a convergent subsequence of $(\lambda_\ell^*)$.}
By \Cref{lem:uniform_bdd} and the Banach-Alaoglu theorem, the sequence of optimizers $(\lambda_\ell^*)$ is contained within a weak$^*$ compact ball $B_\rho(\cl D^{\bnd})$. Therefore, there exists a subsequence, which we re-index as $(\lambda_\ell^*)$, and a limit point $\tilde{\mu} \in B_\rho(\cl D^{\bnd})$ such that $\lambda_\ell^* \xrightarrow{w^*} \tilde{\mu}$.

\noindent\textbf{2. Feasibility and optimality of $\lambda_\ell^*$ in \eqref{eq:relax_obj}.}
    To show $\tilde{\mu} \in \cl L$, 
    let $q \in \cl Q(g_i)$. For $\ell \ge \deg(q)$, the feasibility of $\lambda_{\ell,i}^*$ implies $\apply{\lambda_{\ell,i}^*}{q} \ge 0$. By weak$^*$ convergence, the limit is non-negative: 
    $$\lim_{\ell \to \infty} \apply{\lambda_{\ell,i}^*}{q} = \apply{\tilde{\mu}_i}{q} \ge 0.$$
    Since $q$ was arbitrary, we have $\tilde{\mu}_i \in \cl L(g_i)$, and thus $\tilde{\mu} \in \cl L$.

    To show $t - \apply{\tilde{\mu}}{h} \in K$, note that the feasibility $\lambda_\ell^*$ implies that $t^\ell - \apply{\lambda_\ell^*}{h^\ell} \in K^\ell$. 
    By definition of $K^\ell$, there exists $z_\ell \in K$ such that  $z_{\ell,j} = t_j - \apply{\lambda_\ell^*}{h_j}$ for all $j \in J_\ell$. 
    For a fixed $j \in J$ we have by the weak$^*$ convergence of $\l_\ell^*$ that
    $$
    \lim_{\ell \rightarrow \infty} z_{\ell, j} = \lim_{\ell \rightarrow \infty} t_j - \apply{\lambda_\ell^*}{h_j} = t_j - \apply{\tilde{\mu}}{h_j}
    $$
    Since this holds for any $j \in J$, $(z_\ell)$ converges to $t - \apply{\tilde{\mu}}{h}$ in the product topology. The result follows by the closure property of $K$.    

    Finally, to show optimality, note that from \Cref{cor:gmp_convergence}, we have  $\lim_{\ell \to \infty} \pop_\ell^* = \pop^*$. By definition, $\pop_\ell^* = \apply{\lambda_\ell^*}{f}$. Taking the limit along a convergent subsequence, we have
    $$
    \apply{\tilde{\mu}}{f} = \lim_{\ell \to \infty} \apply{\lambda_\ell^*}{f} = \lim_{\ell \to \infty} \pop_\ell^* = \pop^*.
    $$
    Since $\tilde{\mu}$ is feasible and attains the minimum, it is optimal.

\noindent \textbf{3. Convergence of the full sequence.}
We have shown that the sequence $(\lambda_\ell^*)$ lies in a compact set and that any of its limits must be a minimizer of $\pop^*$. The minimizer is unique by assumption, so the sequence has only one limit point: $\mu^*$. Therefore, the entire sequence $(\lambda_\ell^*)$ converges to $\mu^*$.
\end{proof}
\erev

\section{Application to symmetric tensor decomposition}\label{sec:applications}

The tensor decomposition problem consists in decomposing a tensor of order $d$ into a weighted sum of a minimal number of tensor products of $d$ vectors. When the tensor is symmetric (i.e. invariant by permutation of the indices), the symmetric decomposition is the weighted sum of a minimal  number of tensor products of the same vector.
This can be cast into a generalized moment problem on polynomials as follows.

We identify symmetric tensors $F\in S^{d}(\RR^{n+1})$ of order $d$, with polynomials of degree $\le d$ in $n$ variables: $F(x_1, \ldots, x_n)=F(\bm x)= \sum_{|\a|\le d} F_{\a} \bm x^{\a}$ where $F_{\a}$ are the coefficients of the symmetric tensor $F$ in the basis $(\bm e_0^{d-|\a|} \bm e_1^{\a_1} \cdots \bm e_n^{\a_n})_{|\a| \le d}$ of $S^{d}(\RR^{n+1})$ and $(\bm e_0, \bm e_1, \ldots, \bm e_n)$ is the canonical basis of $\RR^{n+1}$. 

Decomposing the symmetric tensor $F$ consists in writing $F(\bm x)$ as a weighted sum of a minimal number $r$ of $d^{\mathrm{th}}$ power of degree-$1$ polynomials:
$$
F(\bm x) = \sum_{i=1}^{r} \omega_i\, (\xi_{i,0} + \xi_{i,1} x_1 + \cdots + \xi_{i,n} x_n)^d
$$
with $\omega_i, \, \xi_{i,j} \in \RR$.
Hereafter, we are considering decompositions with $\xi_{i,0} \ne 0$ (which we can achieve by a generic change of coordinates in $\RR^{n+1}$). By scaling, we can then assume that $\xi_{i,0}=1$ and the decomposition we are considering is of the form 
\begin{equation}\label{eq:decomp}  
F(\bm x) = \sum_{i=1}^{r} \omega_i\, (1 + \apply{\xi_i}{\bm x})^d
\end{equation}
with $\omega_i \in \RR$, $\xi_i=(\xi_{i,1}, \ldots, \xi_{i,n}) \in \RR^n$. Since the points $\xi_i$ of such a decomposition are inherently bounded, we, without loss of generality, rescale the problem so that they lie on the unit ball: $S = \{ x \in \RR^n \mid 1 - \norm{x}^2 \ge 0\}$.

To transform the tensor decomposition problem into a moment problem, we introduce the \emph{apolar product}: $\forall f= \sum_{|\a|\le d} f_\a \bm x^{\a}$, $g= \sum_{|\a|\le d} g_\a \bm x^{\a} \in \rng_{\le d}$, 
$$
\apply{f}{g}_d = \sum_{|\a|\le d} \binom{d}{\a}^{-1} f_{\a}\,g_{\a}
$$
with $\binom{d}{\a}= \frac {d!}{ (d-|\a|)! \a_1! \cdots \a_n!}$. 
We directly verify that $\forall g \in \rng_{\le d}, \forall \ \xi \in \RR^{n}$,
\begin{equation}\label{eq:apolar}
\apply{(1 + \apply{\xi}{\bm x})^d}{g}_d= \apply{\delta_{\xi}}{g} = g(\xi) 
\end{equation}
where $\delta_{\xi} \in \cl M(S)$ is the Dirac measure at $\xi\in S$.

Therefore, $F$ has a decomposition of the form \eqref{eq:decomp} if and only if $\Lambda_{F}: p \in \rng_{\le d} \mapsto \apply{F}{p}_d\in \RR$ coincides with $\mu =\sum_{i=1}^{r} \omega_i\, \delta_{\bm\xi_i}$ on $\rng_{\le d}$.

\subsection{Positive symmetric tensor decomposition}

We consider first the case where $F$ admits a \emph{positive} decomposition of the form 

\begin{equation} \label{eq:decomp pos}
    F(\bm x) = \sum_{i=1}^{r} \omega_i (1+ \apply{\bm \xi_i}{\bm x})^{d} 
    \quad \textup{with } \omega_i\in \RR_{+} \textup{ and } \bm \xi_i \in S \subset \RR^n. 
\end{equation}

Here, we are searching for a positive measure $\mu= \sum_{i=1}^{r} \omega_i \delta_{\bm\xi_i} \in \cl M_{+}(S)$, which satisfies
$$
\apply{\mu}{\bm x^{\a}} = \apply{F}{\bm x^{\a}}_{d}= \binom{d}{\a}^{-1} F_{\a} \quad \textup{ for } |\a| \le d,
$$
where $ \binom{d}{\a}= \frac{d!}{(d-|\a|)! \a_1! \cdots \a_n!}$.
This yields a \emph{positive decomposition} of $F$, which can be reformulated into the following.
\par
\noindent{}{\bf Generalized Moment Problem for Positive Symmetric Tensor Decomposition:} For $F= \sum_{|\a|\le d} F_{\a}\, \bm x^{\a}$, solve the optimization problem
\begin{equation}\label{eq:tensordec_positive_mu}
        \begin{split}
            \ \  \textup{argmin}\ \  & \sharp(\mu) \\
            \ \ \ \ \textnormal{s.t.} \ \ &\apply{\mu }{\bm x^{\a}} = \binom{d}{\a}^{-1} F_{\a} \textup{ for } |\a|\le d,\\
            & \mu\in \cl M_+(S)
        \end{split}
\end{equation}
where $\sharp(\mu)$ is the number of points in the support of $\mu$. Since this objective function is hard to compute, we relax it into a linear cost function 
$\apply{\mu}{\Psi}$ where 
$\Psi \in \Sigma^2_{2d'}$  with $2d' > d$, 
which aims to encourage solutions where the underlying measure has small total mass, heuristically encouraging a low-rank decomposition. If we take $\Psi= \sum_{|\a|\le d'} \bm x^{2\a}$, we can check that $\apply{\l}{\Psi}$ is the nuclear norm (or sum of the singular values, or trace) of the semi-definite moment matrix $H_{\l}^{d',d'}$ of $\l$ in degree $(d',d')$. The nuclear norm is a well-known heuristic objective in the matrix rank minimization community \cite{fazel2002matrix}. It has been observed to produce low-rank solutions in practice, and induce an exact relaxation under the restricted isometry property \cite{recht2010guaranteed}.

The resulting primal generalized moment problem reads
\begin{equation}\label{eq:tensorde_ex}
        \begin{split}
            \ \  \textup{min}\ \  & \apply{\l}{\Psi}  \\
            \ \ \ \ \textnormal{s.t.} \ \ &\apply{\l}{\bm x^{\a}}=  \binom{d}{\a}^{-1} F_{\a} \textup{ for } |\a|\le d,\\
            & \l \in \cl M_+(S).
        \end{split}
\end{equation}
This program will be relaxed into a hierarchy of finite-dimensional pseudo-moment convex optimization programs denoted \eqref{eq:tensorde_ex}$_\ell$, replacing the constraint $\l\in \cl M_{+}(S)$ by $\l \in \cl L_\ell(S)$ for $\ell \ge 2d'$. 
The existence of a solution of the GMP is guaranteed by the existence of the decomposition \eqref{eq:decomp pos}.

The dual problem of \eqref{eq:tensorde_ex} is written in terms of polynomials as
\begin{equation}\label{eq:tensorpr_ex}
        \begin{split}
            \ \  \sup\ \  & \apply{F}{V}_{d}  \\
            \ \ \ \ \textnormal{s.t.} \ \ & V\in \rng_{\le d}, \\
            &\Psi - V \in \pos(S)
        \end{split}
\end{equation}
where $V=\sum_{|\alpha|\le d} v_{\alpha} \bm x^{\alpha}$ is the polynomial associated with the dual variables $v_{\a}$ and $\pos(S)$ is the cone of polynomials, which are non-negative on $S$.
By the property \eqref{eq:apolar} of the apolar product, we have $\apply{F}{V}_d = \sum_{i=1}^{r} \omega_i V(\xi_i)$.

\begin{proposition}\label{prop:attain}
Let $I_{\xi}=\{p \in \rng\mid p(\xi_i)=0 \textup{ for } i=1, \ldots, r\}$ be the vanishing ideal of the points $\xi_i$ of the decomposition \eqref{eq:decomp pos}. Assume that $\rng_{\le d}$ contains a basis $B$ of $\rng/I_{\xi}$. Then \Cref{as:x-full} ($S$-fullness) and  \Cref{as:dual_optimizer} (\eqref{eq:tensorpr_ex} attains its supremum) are satisfied.
\end{proposition}
\begin{proof}
The $S$-fullness \Cref{as:x-full} is satisfied since $1= \bm x^{\bm 0}>0$ on $S$. By \Cref{lem:compactness_p_star} and \Cref{prop:slaters_d_star}, the values of \eqref{eq:tensorde_ex} and  \eqref{eq:tensorpr_ex} are equal. 
We consider the subproblem 
\begin{equation}\label{eq:dual on B}
\begin{split}
            \ \  \sup\ \  & \apply{F}{V}_{d}  \\
            \ \ \ \ \textnormal{s.t.} \ \ & V\in \vspan{B}, \, \Psi - V \in \pos(S)
\end{split}
\end{equation}
which is bounded by $\dop^*$. Assume, for contradiction, that an optimizing sequence $(V_k)$ for this subproblem is not bounded, i.e., $\|V_k\| \to \infty$ (where $\| \cdot\|$ is the apolar norm). Let $\tilde{V}_k = V_k / \|V_k\|$, so $\|\tilde{V}_k\| = 1$. By the Bolzano-Weierstrass theorem, there exists a convergent subsequence (re-indexed as  $\tilde{V}_k$) such that $\tilde{V}_k \to \tilde{V}$ with $\|\tilde{V}\| = 1$.
Since 
$\Psi/ \norm{V_k} - \tilde{V}_k \ge 0  \ \textup{on } S$ 
and $\norm{V_k}\to \infty$, we have 
$$ 
- \tilde{V} \ge 0 \ \ \textup{on } S, \textup{ and } \apply{F}{\tilde{V}}_d= \sum_{i=1}^{r} \omega_i \tilde{V}(\xi_i)= 0. 
$$
This implies that $\tilde{V}(\xi_i)=0$ for $i=1, \ldots, r$ and $\tilde{V}\in I_{\xi}$. As $B$ is a basis of $\rng/I_{\xi}$ and $\tilde{V}\in \vspan{B}$, we deduce that 
$\tilde{V}=0$, contradicting $\|\tilde{V}\|=1$.
Thus, the optimizing sequence $(V_k)$ must be bounded and has a convergent subsequence to $V^*$, which is a maximizer of \eqref{eq:dual on B}. 

Then it is also a maximizer of \eqref{eq:tensorpr_ex} since $\rng_{\le d}= \vspan{B}\oplus (I_{\xi} \cap \rng_{\le d})$ and
$\forall p \in (I_{\xi} \cap \rng_{\le d})$, $\apply{F}{p}_d = \sum_{i=1}^{r} \omega_i p(\xi_i)=0$.
Consequently, \eqref{eq:tensorpr_ex} attains its supremum at $V^*$. This concludes the proof.
\end{proof}
We deduce the following corollary:
\begin{theorem}\label{thm:tensordec_pos}
If $S= \{x \in \RR^n\mid 1-\norm{x}^2 \ge 0\}$, $F = \sum_{i=1}^{r} \omega_i (1+ (\xi_i,\bm x))^{d}$ with $\omega_i\in \RR_{+}$ {and} $\bm\xi_i \in S$ and the quotient algebra $\cl A_{\xi}= \rng/I_{\xi}$ by the vanishing ideal of the points $\xi$ admits a basis in degree $\le d$, then 
the optima $\pop^*$ (resp. $\pop^*_{\ell}$) of 
\eqref{eq:tensorde_ex} (resp. \eqref{eq:tensorde_ex}$_{\ell}$) 
and the optima $\dop^{*}$ (resp. $\dop^*_{\ell}$) of the dual problem \eqref{eq:tensorpr_ex}  (resp. \eqref{eq:tensorpr_ex}$_{\ell}$) are reached, $\pop^*=\dop^*$ and we have, for $\ell$ sufficiently large,
$$
0 \le \pop^* - \pop_\ell^* \le  \dop^* - \dop^*_\ell \le  \kappa\, \ell^{-2},
$$
with 
$$
\kappa \le  \gamma\, (2 \Psi_{\max} + \norm{V^*}_{1}),
$$ 
$\gamma$ given in  \Cref{thm:putinar_effective}, Item 4, 
$V^*= \sum_{|\a|\le d} v^*_{\a} \bm x^\a$ a maximizer of $\dop^*$,
$\norm{V^*}_1=$ $\sum_{|\a|\le d} |v^*_{\a}|$ and
$\Psi_{\max} = \max_{x\in S} \Psi(x)$.
\end{theorem} 
\begin{proof}
Since  the unit ball $S$ is defined by $1 - \norm{\bm x}^2\ge 0$, the problem satisfies the Archimedean  \Cref{as:archimedean} trivially.  By \Cref{prop:attain},
\Cref{as:dual_optimizer} and \Cref{as:x-full} are satisfied.
We thus apply \Cref{thm:gmp_convergence} with the exponent $\theta = 2$ taken from \Cref{thm:putinar_effective}, and the relation $\max_{x \in S} |x^{\alpha}|\le 1$ for $|\a|\le d$ to conclude the proof.
\end{proof}

\begin{remark}
In \cite[Proposition 3.6]{Nie2014}, a condition of strict $S$-positivity is used to show the existence of a maximizer of \eqref{eq:tensorpr_ex}. It corresponds to the case where the basis $B$
contains the set of all monomials of degree $d$. The condition in \Cref{prop:attain} is less restrictive, and allows  more general decompositions of the form \eqref{eq:decomp pos}.
\end{remark}

\begin{remark}\label{rem:uniqueness}
In the case $\Psi\in \Sigma^2_{2d'}$ is generic (i.e. $\Psi\in  \Sigma^2_{2d'} \setminus \Theta$, for a subset $\Theta \subseteq \Sigma^2_{2d'}$ having zero Lebesgue measure) with $2d' = \ell$, it is possible to prove that there is a  unique degree $2d'$ truncated minimizer $(\l^*)^{(2d')}$ for $\lambda^*\in \cl L_{\ell}(S)$ any minimizer of  \eqref{eq:tensorde_ex}$_\ell$. We use the fact that the set of singular normal vectors of a convex body has a zero Lebesgue measure (\cite[Theorem 2.2.11]{schneider2013convex}). See also \cite[Proposition 5.2, (i)]{Nie2014} for the uniqueness of the optimum for moment sequences of measures truncated in degree $2d'$. 
\end{remark}

\begin{example}[$n=2$, $d=4$, $r=4$.]\label{example:pos}We consider the tensor of dimension $n+1 = 3$ and order $d=4$, corresponding to the derivative with respect to the first variable $x_0$ of the first example in \cite{brachat2010symmetric}, divided by 5:
\begin{equation*}
    \begin{split}
    F(\textbf{x}) = & \ 88518\,x_2^4 - 309888\,x_1x_2^3 + 483408\,x_1^2x_2^2 - 165888\,x_1^3x_2 + 166368\,x_1^4 \\
& - 23664\,x_0x_2^3 + 66528\,x_0x_1x_2^2 - 88992\,x_0x_1^2x_2 - 13824\,x_0x_1^3 \\
& + 4932\,x_0^2x_2^2 - 3456\,x_0^2x_1x_2 + 7632\,x_0^2x_1^2 \\
& + 144\,x_0^3x_2 - 96\,x_0^3x_1 + 38\,x_0^4.
    \end{split}
\end{equation*}
We dehomogenize $F$ by substituting $x_0=1$ to obtain a polynomial in variables $\bm x = (x_1, x_2)$. To transform the search space to the unit ball, we take a rescaling factor of $20$. We then solve the SDP relaxation \eqref{eq:tensorde_ex}$_\ell$ for the rescaled problem. We set the relaxation order to $\ell=12$ and use the nuclear norm heuristic as the objective, with $\Psi(\bm x) = \sum_{|\a|\le 6} \bm x^{2\a}$. The problem is solved using the Julia package \href{https://github.com/AlgebraicGeometricModeling/MomentPolynomialOpt.jl}{MomentPolynomialOpt.jl}. The solver returns a moment sequence, from which we extract the weights
$$
\omega = \begin{bmatrix}
  5.000000920403701 \\
  3.000000015538793 \\
  15.000004556660205 \\
  14.999994501024343
\end{bmatrix}
$$
along with the scaled support points $\Xi$:
$$
\Xi = \begin{bmatrix}
  1.0  &  1.0  &  1.0  &  1.0 \\
 -0.6 & 0.6  & -0.1  &  0.1 \\
 -0.15  & -0.65 &  0.15  &  0.15
\end{bmatrix}.
$$
By the appropriate re-scaling --- multiplication of second and third row of $\Xi$ by $20$, and division of $\omega$ by $5$ --- we observe the same decomposition as in \cite[Example 1]{brachat2010symmetric}. 
This example also shows that the hierarchy of relaxations \eqref{eq:tensorde_ex} is (numerically) exact, i.e. it achieves the optimum at a finite order ($\ell = 12$ here) of relaxation. This phenomenon has been observed for polynomial optimization problems, for which regularity or generality assumptions imply the SoS and moment exactness of the relaxation \cite{baldi:exactness:2024,nie_optimality_2014,Nie2014}. 
This example also shows that the polynomial bounds on the rate of convergence of Theorem \ref{thm:gmp_convergence} can be pessimistic.
\end{example}

\subsection{Real symmetric tensor decomposition}\label{sec:2d_tensorexample}
Let us consider now the case where $F$ admits a decomposition of the form 
\begin{equation} \label{eq:decomp real}
    F(\bm x) = \sum_{i=1}^{r} \omega_i (1+ \apply{\bm \xi_i}{\bm x})^{d} 
    \quad \textup{with } \omega_i\in \RR \textup{ and } \bm \xi_i \in S \subset \RR^n.
\end{equation}
We write $\mu = \mu_{+} - \mu_{-}$ where $\mu_+ =\sum_{\omega_i >0} \omega_i\, \delta_{\bm\xi_i} \in \cl M_{+}(S) $ and $\mu_{-} = \sum_{\omega_i<0} - \omega_i\, \delta_{\bm\xi_i} \in \cl M_{+}(S)$.
We obtain the following reformulation of the symmetric tensor decomposition for points in $S = \{x \in \RR^n \mid 1 - \norm{x}^2 \ge 0\}$.
\par
\noindent{}{\bf Generalized Moment Problem for Symmetric Tensor Decomposition:} For $F= \sum_{|\a|\le d} F_{\a}\, \bm x^{\a}$, solve the optimization problem
\begin{equation}\label{eq:tensordec_mu}
        \begin{split}
            \ \  \textup{argmin}\ \  & \sharp(\mu_+ + \mu_-) \\
            \ \ \ \ \textnormal{s.t.} \ \ &\apply{[\mu_+ , \mu_-]}{[\bm x^{\a}, - \bm x^{\a}]}=  \binom{d}{\a}^{-1} F_{\a} \textup{ for } |\a|\le d,\\
            & [\mu_+, \mu_-] \in \cl M_+(S)^2.
        \end{split}
\end{equation}
We replace the objective function by the proxy $\apply{\l_+}{\Psi_+}  +  \apply{\l_-}{\Psi_-}$, where $\Psi_+, \Psi_- \in \Sigma^2_{2d'}$  with $2d' > d$. Note that \Cref{as:x-full} cannot be satisfied, since both $\mathbf x^\alpha$ and $-\mathbf x^\alpha$ appear as columns in the matrix $h$. To circumvent this, we constrain the total variation: 
$$
\apply{\l_+}{1} + \apply{\l_-}{1} \le L, \ \ \ L \ge 0.
$$ 
This constraint corresponds to having a finite-weight  decomposition. We further impose that $L$ satisfies
\begin{equation}\label{eq:L_characterisation}
L > \sum_{i=1}^r |\omega_i| > 0
\end{equation}
to ensure the feasibility of the underlying measure from \eqref{eq:decomp real}.
The resulting primal generalized moment problem reads
\begin{equation}\label{eq:tensordec_real_ex}
        \begin{split}
            \ \  \textup{min}\ \  & \apply{\l_+}{\Psi_+}  +  \apply{\l_-}{\Psi_-} \\
            \ \ \ \ \textnormal{s.t.} \ \ &\apply{\l_+}{\bm x^{\a}} - \apply{\l_-}{\bm x^{\a}}=  \binom{d}{\a}^{-1} F_{\a} \textup{ for } |\a|\le d,\\
            & \apply{\l_+}{1} + \apply{\l_-}{1} \le L,\\
            & \l_+, \l_- \in \cl M_+(S).
        \end{split}
\end{equation}
This program will be relaxed into a hierarchy of finite-dimensional convex optimization programs, denoted by \eqref{eq:tensordec_real_ex}$_{\ell}$, by replacing the constraint $\l_+, \l_- \in \cl M_+(S)$ by $\l_+, \l_- \in \cl L_\ell(S)$ for $\ell\ge 2d'$.
The dual problem to \eqref{eq:tensordec_real_ex} is written as
\begin{equation}\label{eq:tensordec_real_ex_dual}
\begin{split}
    \ \  \textup{sup}\ \  & \langle F, V \rangle_d - Lu \\
    \ \ \ \ \textnormal{s.t.} \ \ &\Psi_+ - V  + u \in \pos(S),\\
    &\Psi_- + V + u \in \pos(S),\\
    & V \in \rng_{\le d}, \ u \in \RR_+.
\end{split}
\end{equation}
\begin{proposition}\label{prop:sup_attainment_reals}
The supremum in \eqref{eq:tensordec_real_ex_dual} is attained. 
\end{proposition}
\begin{proof}
We prove this proposition in two steps:\\
\textbf{1. For fixed $u \in \RR_+$, the supremum over $V \in \rng_{\le d}$ is attained.}
Let $u \ge 0$ be fixed, and assume for contradiction that an optimizing sequence $(V_k)$ for the resulting problem is not bounded, that is, $\|V_k\| \to \infty$. Let $\tilde{V}_k = V_k / \|V_k\|$, so $\|\tilde{V}_k\| = 1$. By the Bolzano-Weierstrass theorem, there exists a convergent subsequence (re-indexed as $\tilde{V}_k$) such that $\tilde{V}_k \to \tilde{V}$ with $\|\tilde{V}\| = 1$.
The feasibility of $V_k$ implies that the following inequalities hold on $S$:
$$ \Psi_+ - V_k  + u\ge 0 \quad \text{and} \quad \Psi_- + V_k + u\ge 0. $$
Dividing by $\|V_k\|$ and letting $k \to \infty$, we obtain
$$ 
0 \ge \tilde{V} \ge 0 \quad \text{on } S, 
$$
since $\Psi_+/\|V_k\| \to 0$ and $\Psi_-/\|V_k\| \to 0$, and $u$ is fixed. Thus, the polynomial $\tilde{V}$ is identically zero on $S$. Since $S$ has a non-empty interior, $\tilde{V}$ must be the zero polynomial. In particular, $\|\tilde{V}\| = 0$, which is in contradiction with $\|\tilde{V}\| = 1$.
Therefore, the optimizing sequence $(V_k)$ must be bounded. Consequently, it has a convergent subsequence whose limit is a polynomial $V_u\in \rng_{\le d}$ which attains the supremum \eqref{eq:tensordec_real_ex_dual} for any fixed $u\in \RR_{+}$. 

\noindent\textbf{2. The optimizing sequence for $u \in \RR_+$ is bounded.}
 Assume for contradiction that an optimizing sequence $(u_k) > 0$ is not bounded, that is, $u_k \to \infty$. 
The feasibility of each pair $(u_k, V_{u_k})$ implies that the following inequalities hold on $S$:
$$ \Psi_+ - V_{u_k}  + u_k \ge 0 \quad \text{and} \quad \Psi_- + V_{u_k} + u_k \ge 0. $$
Dividing by $u_k$, we obtain
$$
\Psi_+/u_k  + 1 \ge \tilde{V}_k  \ge -\Psi_-/u_k - 1,
$$
where $\tilde{V}_k=  V_{u_k} / u_k$.
This shows that  $\norm{\tilde{V}_{k}}$ is a bounded sequence with respect to $k$. By the Bolzano-Weierstrass theorem, $(\tilde{V}_k)$ has a convergent subsequence, whose limit is denoted by $\tilde{V}$. Taking $k\rightarrow \infty$ along this subsequence, we have
\begin{equation}\label{eq:sup_attainment_reals_1}
1 \ge \tilde{V}(x)  \ge -1  \ \ \forall \, x \in S.
\end{equation}
Dividing by $u_k$ and taking $k\rightarrow \infty$ in the objective function, we obtain
$$
\langle F, \tilde{V} \rangle_d - L = 0.
$$
By properties of the apolar product \eqref{eq:apolar}, this can also be written as
\begin{equation}\label{eq:sup_attainment_reals_2}
\sum_{i=1}^{r^+} \omega^+_i \tilde{V}(\xi^+_i) - \sum_{j=1}^{r_-} \omega^-_j \tilde{V}(\xi^-_i) - L = 0,
\end{equation}
where $r^+$ and $r^-$ are the number of points $\{\xi_i^+\}_{i=1}^{r^+}$ (resp. $\{\xi_i^-\}_{j=1}^{r^-}$) in the decomposition  \eqref{eq:decomp real} such that $\omega_i = \omega_i^{+}>0$ (resp. $\omega_{i}= -\omega_{i}^{-}<0$).  Combining \eqref{eq:sup_attainment_reals_1} and \eqref{eq:sup_attainment_reals_2}, we have
$$
\sum_{i=1}^{r^+} \omega^+_i + \sum_{j=1}^{r_-} \omega^-_j \ge L.
$$
The above is in contradiction with \eqref{eq:L_characterisation}, thus, the optimizing sequence $(u_k)$ must be bounded. Since the boundedness of $u \in \RR_+$ enforces the boundedness of $V \in \rng_{\le d}$ via the feasibility constraints, the supremum of \eqref{eq:tensordec_real_ex_dual}  is attained.
\end{proof}

\begin{theorem}\label{thm:tensordec_real}
If $S= \{x \in \RR^n\mid 1-\norm{x}^2 \ge 0\}$, $F = \sum_{i=1}^{r} \omega_i (1+ (\xi_i,\bm x))^{d}$ with $\omega_i\in \RR$ {and} $\bm\xi_i \in S$ and $L> \sum_i |\omega_i|$, then 
the optima $\pop^*$ (resp. $\pop^*_{\ell}$) of 
\eqref{eq:tensordec_real_ex} (resp. \eqref{eq:tensordec_real_ex}$_{\ell}$) 
and the optima $\dop^{*}$ (resp. $\dop^*_\ell$) of the dual problems \eqref{eq:tensordec_real_ex_dual} (resp. \eqref{eq:tensordec_real_ex_dual}$_\ell$) are reached, 
$\pop^*=\dop^*$ and we have, for $\ell$ sufficiently large,
$$
0 \le \pop^* - \pop_\ell^* \le \dop^* - \dop^*_\ell \le \kappa\, \ell^{-2},
$$
with 
$$
\kappa \le  \gamma\, L\, (2 \Psi_{\max} + \norm{V^*}_{1} + u^{*}),
$$ 
$\gamma$ given in  \Cref{thm:putinar_effective}, Item 4, 
$(V^*, u^*)$ a maximizer of $\dop^*$, 
$\Psi_{\max}= $ \\ $ \max_{x\in S} \{\Psi_{+}(x), \Psi_{-}(x)\}$.
\end{theorem}
\begin{proof}
The feasible set of the truncated moment problem is a nonempty (due to \eqref{eq:decomp real}), compact and convex set.
Since $S$ is the unit ball, the problem satisfies the Archimedean condition \Cref{as:archimedean} trivially.
\Cref{as:x-full} ($S$-fullness) is satisfied due to the total variation constraint: choose vector $w$ that is non-zero only on the entry corresponding to this constraint. 
\Cref{as:dual_optimizer} is satisfied by \Cref{prop:sup_attainment_reals}.
We thus apply \Cref{thm:gmp_convergence} with the exponent $\theta = 2$ taken from \Cref{thm:putinar_effective}, and \Cref{cor:gmp_convergence} and use the relation $\max_{x \in S} |x^{\alpha}|\le 1$ for $|\a|\le d$ to conclude the proof.
\end{proof}

\begin{remark}
One can prove the uniqueness of the truncation  $[(\l^*_+)^{(2d')},(\l^*_-)^{(2d')}]$ in degree $2d'$  for a minimizer $[\l^*_+,\l^*_-]$ of \eqref{eq:tensordec_real_ex}$_{\ell}$ for $\Psi_+$,  $\Psi_-$ \emph{generic} in  $\Sigma^2_{2d'}$, using that the set of singular normal vectors of the feasibility set of \eqref{eq:tensordec_real_ex}$_{\ell}$ has a zero Lebesgue measure (see \cite[Theorem 2.2.11]{schneider2013convex} and \Cref{rem:uniqueness}).
\end{remark}

\begin{example}[$n=3$, $d=4$, $r=7$.]\label{example:real}In this example, we consider the following symmetric tensor given by 
\begin{footnotesize}
\begin{equation*}
\begin{split}
&F(\textbf{x}) =
-0.2489979301598193 \, x_3^4 - 0.5714471586952264 \, x_2 x_3^3 - 0.5050309495726817 \, x_2^2 x_3^2   \\
&+ 0.20686591014734546 \, x_2^3 x_3 + 0.35848687448905797 \, x_2^4 + 0.6290084096294283 \, x_1 x_3^3   \\
&+ 1.5863914541246662 \, x_1 x_2 x_3^2 + 0.09102584099156069 \, x_1 x_2^2 x_3 - 0.7995202943157504 \, x_1 x_2^3 \\
& - 1.000984626080145 \, x_1^2 x_3^2 - 0.7330971731383541 \, x_1^2 x_2 x_3 + 0.48952697145340573 \, x_1^2 x_2^2 \\
&+ 0.3654972483140739 \, x_1^3 x_3 + 0.04180994666887122 \, x_1^3 x_2 - 0.06496525932745015 \, x_1^4   \\
&+ 0.8614765507794534 \, x_0 x_3^3 + 2.518149761933887 \, x_0 x_2 x_3^2 + 0.11198414551091718 \, x_0 x_2^2 x_3   \\
&- 1.2824172455614082 \, x_0 x_2^3 - 3.198429373199247 \, x_0 x_1 x_3^2 - 2.280206948221652 \, x_0 x_1 x_2 x_3  \\
&+ 1.51085249826687 \, x_0 x_1 x_2^2 + 1.8431809395581764 \, x_0 x_1^2 x_3 + 0.38190508335832307 \, x_0 x_1^2 x_2 \\ 
&- 0.6712764568310912 \, x_0 x_1^3- 2.2452396434195863 \, x_0^2 x_3^2 - 1.3839699114386361 \, x_0^2 x_2 x_3   \\
&+ 1.4125957148603887 \, x_0^2 x_2^2 + 1.4670698710382422 \, x_0^2 x_1 x_3- 0.015953647647009017 \, x_0^2 x_1 x_2   \\
&- 0.361365758387135 \, x_0^2 x_1^2 + 1.253344034777953 \, x_0^3 x_3 + 0.11846876725173883 \, x_0^3 x_2 \\
&- 2.1207614308184404 \, x_0^3 x_1 + 0.6141543193928954 \, x_0^4
\end{split}
\end{equation*} 
\end{footnotesize}
We dehomogenize $F$ by substituting $x_0=1$ to obtain a polynomial in variables $\bm x = (x_1, x_2, x_3)$. To transform the search space to the unit ball, we take a rescaling factor of $2$. We then solve the SDP relaxation \eqref{eq:tensorde_ex}$_\ell$ for the rescaled problem. We set the relaxation order to $\ell = 12$ and use the nuclear norm heuristic as the objective, with $\Psi(\bm x) = \sum_{|\a|\le 6} \bm x^{2\a}$. Solving the relaxation yields a moment matrix corresponding to a rank-$16$ decomposition.

However, by the Alexander-Hirschowitz Theorem (see e.g. \cite{brambilla2008alexander}), tensors of dimension $n + 1 = 4$ and order $d = 4$ have a generic rank $9$. In fact, this tensor has a rank $7$ decomposition, and we locate it exactly by solving \eqref{eq:tensorde_ex}$_\ell$ augmented with auxiliary constraints, which are introduced presently.

The degree 2 by degree 2 Hankel matrix $H^{2,2}$ also known as $(2,2)$-\emph{catalecticant} of $F$ \cite{pucci1998veronese}, fully determined by the coefficients of the tensor, is
\[
\scalebox{0.5625}{$
\begin{array}{c|cccccccccc}
& 1 & x_3 & x_2 & x_1 & x_3^2 & x_3x_2 & x_2^2 & x_3x_1 & x_2x_1 & x_1^2 \\
\hline
1 & 
0.614154 & 0.313336 & 0.0296172 & -0.53019 & -0.374207 & -0.115331 & 0.235433 & 0.122256 & -0.00132947 & -0.0602276 \\
x_3 & 0.313336 & -0.374207 & -0.115331 & 0.122256 & 0.215369 & 0.209846 & 0.00933201 & -0.266536 & -0.0950086 & 0.153598 \\
x_2 & 0.0296172 & -0.115331 & 0.235433 & -0.00132947 & 0.209846 & 0.00933201 & -0.320604 & -0.0950086 & 0.125904 & 0.0318254 \\
x_1 & -0.53019 & 0.122256 & -0.00132947 & -0.0602276 & -0.266536 & -0.0950086 & 0.125904 & 0.153598 & 0.0318254 & -0.167819 \\
x_3^2 & -0.374207 & 0.215369 & 0.209846 & -0.266536 & -0.248998 & -0.142862 & -0.0841718 & 0.157252 & 0.132199 & -0.166831 \\
x_3x_2 & -0.115331 & 0.209846 & 0.00933201 & -0.0950086 & -0.142862 & -0.0841718 & 0.0517165 & 0.132199 & 0.00758549 & -0.0610914 \\
x_2^2 & 0.235433 & 0.00933201 & -0.320604 & 0.125904 & -0.0841718 & 0.0517165 & 0.358487 & 0.00758549 & -0.19988 & 0.0815878 \\
x_3x_1 & 0.122256 & -0.266536 & -0.0950086 & 0.153598 & 0.157252 & 0.132199 & 0.00758549 & -0.166831 & -0.0610914 & 0.0913743 \\
x_2x_1 & -0.00132947 & -0.0950086 & 0.125904 & 0.0318254 & 0.132199 & 0.00758549 & -0.19988 & -0.0610914 & 0.0815878 & 0.0104525 \\
x_1^2 & -0.0602276 & 0.153598 & 0.0318254 & -0.167819 & -0.166831 & -0.0610914 & 0.0815878 & 0.0913743 & 0.0104525 & -0.0649653 \\
\end{array}$}
\]
By \eqref{eq:decomp real}, $F(\bm x)= \sum_{i=1}^r \omega_i (1 + \apply{\bm\xi_i}{\bm x})^d$ with $\bm \xi_i\in S$, or equivalently, $\mu_F = \sum_{i=1}^{r} \omega_i  \delta_{\bm \xi_i}$.
We verify that a polynomial $q$ of degree $\le 2$, vanishing at $\bm \xi_i$ is in the kernel of $H^{2,2}$: 
$$
\forall q' \in \rng_{\le 2}, \apply{q'}{H^{2,2}(q)} = \apply{\mu_F}{q\,q'} = \sum_{i=1}^{r} \omega_i q(\bm \xi_i) q'(\bm \xi_i)= 0.
$$
This implies that $q \in \ker H^{2,2}$.

 To constrain the solution of \eqref{eq:tensordec_real_ex}, assuming that all polynomials in $\ker H^{2,2}$ vanish at $\xi_i$\footnote{This is not always the case. It happens when $d$ is big enough compared to the regularity of the points $\bm\xi_1, \ldots, \bm\xi_r$. Specifically, when the Vandermonde $V_{k}$ in the Hankel-Vandermonde factorization $H^{k, l} = V_k\Delta V_{l}^T$ has full column rank.}, we add the constraints $\apply{\l_{\pm}}{q\, p}=0$ for $q\in \ker H^{2,2}$ and $p\in \rng_{\ell -d}$.
We arrive at the following augmented form:
\begin{equation}\label{eq:tensorde_ex_2}
        \begin{split}
            \ \  \textup{min}\ \  & \apply{\l_+}{\Psi_+} + \apply{\l_-}{\Psi_-}  \\
            \ \ \ \ \textnormal{s.t.} \ \ &\apply{\l_+}{\bm x^{\a}} -  \apply{\l_{-}}{\bm x^{\a}}= \binom{d}{\a}^{-1} F_{\a} \textup{ for } |\a|\le d,\\
            &\apply{\l_+}{q\, \bm x^{\beta}} =  0,  \ \ \ \forall \  q \in \ker H^{2,2}, |\b|\le \ell -2, \\\
            &\apply{\l_-}{q\, \bm x^{\beta}} =  0,  \ \ \ \forall \  q \in \ker H^{2,2}, |\b|\le \ell -2, \\
            & \apply{\l_+}{1} + \apply{\l_-}{1} \le L,\\
            & \l_+, \l_- \in \cl L_\ell(S).
        \end{split}
\end{equation}
for $2d' > d$. 

We keep $\Psi_+ = \Psi_- = \sum_{|\a|\le 6} \bm x^{2\a},$ and $\ell = 12$, and solve the above via Julia package \href{https://github.com/AlgebraicGeometricModeling/MomentPolynomialOpt.jl}{MomentPolynomialOpt.jl}. The following vector of weights $\omega$ is retrieved: 
\[
\scalebox{0.85}{$
\omega = \begin{bmatrix}
  -0.4424776570254827\\
-0.8445409705776897\\
-0.1434193311604589\\
0.27399137914217964\\
0.38387835648816826\\
0.8233082986137222\\
0.5634142439124565
\end{bmatrix}$}\]
along with the scaled support points $\Xi$:
\[
\scalebox{0.85}{$
\Xi = \begin{bmatrix}
  1.0    &    1.0     &   1.0    &    1.0   &       1.0    &    1.0     &    1.0  \\
0.637235 &  0.653722  & 0.460376 &  0.517252 &    0.728496 & -0.375355   & 0.457124 \\
-0.727603 &  -0.486279 & -0.255616 & -0.0368457 &  -0.995796 &  0.0720327 & -0.721705 \\
-0.023974  & -0.74095  &  0.726747 & -0.630254 &   -0.253459   & 0.158722 &  -0.141105 
\end{bmatrix}.
$}\]
The reconstructed homogeneous polynomial $\sum_{i=1}^{r} \omega_i \apply{\Xi[:,i]}{\mathbf{x}}^{d}$ corresponding to $\omega$, $\Xi$ differs from $F$
by $5.418\times 10^{-7}$ in the apolar norm. We have an exact relaxation at the order $\ell = 12$ up to numerical errors. Similar to the previous example, we observe that the relaxation achieves exactness at a finite order, and that the convergence rate analysis is overly conservative.
\end{example}

\section{Conclusion}\label{sec:conclusion}
Our primary contributions are twofold: First, we derived explicit polynomial convergence rates for the optimal values of the relaxations to the true optimum of the GMPs with countable moment constraints on vectors of measures, under dual optimum attainment, $S$-fullness and Archimedean conditions. These rates, detailed in \Cref{cor:gmp_convergence}, are geometry-adaptive, influenced by the underlying semi-algebraic sets through parameters obtained from effective Putinar's Positivstellensätze. Second, we extended this analysis to the convergence of the feasibility sets themselves. \Cref{thm:haus_dis} provides polynomial rates for the convergence in Hausdorff distance between the cone of truncated moment sequences of positive measures and the cone of pseudo-moments arising from the SoS relaxations.

We also demonstrated the applicability of our framework to minimal rank symmetric tensor decomposition. We formulated this problem as a GMP, validated the dual optimum attainment, $S$-fullness and Archimedean conditions, consequently deriving explicit error bounds (\Cref{thm:tensordec_pos} and \Cref{thm:tensordec_real}). While our experiments in \Cref{example:pos} and \Cref{example:real} show that the convergence rates can be conservative compared to the finite convergence observed in practice, they nonetheless provide valuable worst-case theoretical guarantees.

Promising future research directions include: (i) applying this framework to new domains, such as optimal control; (ii) identifying conditions for the existence of dual optimizers with application-specific a priori bounds; (iii) investigating the gap between the heuristic objective taken in the examples and the true objective; and (iv) characterizing conditions under which moment relaxations exhibit finite convergence at low orders.


\bibliographystyle{siamplain}
\bibliography{references}

\newpage
\appendix

\section{Linear programming proofs}\label{app:lp_proofs}

\medskip

\textit{Proof of \Cref{lem:compactness_p_star}. } 
Observe firstly that $\Omega$ is closed in $\cl M$. 
Let $\mu_n \xrightarrow{w} \mu$ with $\mu_n \in \Omega$ and $\mu \in \cl M$. We readily have that
$$
    \lim_{n\rightarrow \infty}  (t - \apply{\mu_n}{h}) = t - \apply{\mu}{h} \in K.
$$
since $K$ is closed in $\RR^J$.
Moreover, for any $q\in \Cp$, we have $\apply{\mu_n}{q}\ge 0$, thus $\apply{\mu}{q} = \lim_{n\rightarrow \infty}  \apply{\mu_{n}}{q} \ge 0$.
    The previous statements imply $\mu \in \Omega$, and so $\Omega$ is closed.
    
    The Banach-Alaoglu theorem (see e.g. \cite[Theorem 3.5.16]{ash2014measure}) implies that the closed unit ball ${B}_1(\cl M(S_i))$ in the space $\cl M(S_i)$ 
is weak$^*$ compact, as well as the ball $B_{\kappa}(\cl M(S_i))$ of radius $\kappa$ by rescaling. By Tychonoff's Theorem, the ball $B_{\kappa}(\cl M)$ is also weak$^*$ compact.

Let $b_{\min} \BYDEF \min_{i=1,\ldots,m}(\min_{x\in S_i} b_i(x)) > 0 $, since $b_i>0$ on $S_i$ and $S_i$ is compact. Then, since $w \in K^*$ and $\mu \in \Omega$, we have $\mult{(t - \apply{\mu}{h})}{w} \ge 0$. We deduce that
\begin{equation}\label{eq:t_dot_w}
\mult{t}{w} \ge \mult{(\apply{\mu}{h})}{w} = \apply{\mu}{\mult{h}{w}}= \apply{\mu}{b}
\ge  b_{\min} \sum_{i=1}^m \apply{\mu_i}{1}
\end{equation}
This implies that 
\begin{equation}\label{eq:kappa}
    \sum_{i=1}^m \apply{\mu_i}{1} \le \rho \BYDEF \frac{\mult{t}{w}}{b_{\min}}.
\end{equation}
The operator norm of a nonnegative measure is its mass, so we have $\|\mu\|_{\infty} \le \rho$, and $\mu \in B_{\rho}(\cl M)$. We deduce that $\Omega$ is a closed subset of the weak$^*$ compact set in ${B}_{\rho}(\cl M)$, and is therefore also weak$^*$ compact.
The fact that the value of $\pop^*$ is attained follows by the extreme value theorem. \QED

\textit{Proof of \Cref{prop:slaters_d_star}. }
    By taking the following series of substitutions, \eqref{eq:anderson_IP} can be realized as $-\dop^*$, the problem identical to \eqref{eq:gmp_dual_obj} except for the sign of the objective:\\
    \begin{tabular}{p{0.45\textwidth} p{0.45\textwidth}}
    \begin{enumerate}
        \item $v \rightarrow x$,
        \item $\mult{h}{v} \rightarrow Ax$,
        \item $t \rightarrow c $,
        \item $\sup \, \mult{t}{v} \rightarrow -\inf\, \mult{c}{-v}$,
    \end{enumerate}
    &
    \begin{enumerate}
        \setcounter{enumi}{4}
        \item $\Cp \rightarrow Q$,
        \item $K^* \rightarrow P$,
        \item $b \rightarrow -f$.
    \end{enumerate}
    \end{tabular}\\
    Specifically, we obtain
    \begin{equation*}
        \begin{split}
            -\dop^* = \ \ &\textnormal{inf} \ \ \mult{t}{v} \\
            &\textnormal{s.t.} \ \ \mult{h}{v} + f \in \Cp, \\
            &v \in K^*.
        \end{split}
    \end{equation*}
    Since $K$ is closed, $K^{**} = K$ (see e.g. \cite[p. 121]{rockafellar1997convex}, \cite[p. 53]{boyd2004convex}) and by the Riesz representation theorem $(\Cp)^* = \Mp$. It is then readily verified with reference to the above substitutions and  \eqref{eq:anderson_IP*} that the dual to the above problem is $-\pop^*$. The weak duality condition is satisfied: $-\dop^* \ge -\pop^*$, and hence by \Cref{lem:compactness_p_star}, $-\dop^*$ has a finite infimum. To satisfy \Cref{con:slaters}, it remains to locate a vector $v_0 \in K^*$ such that $\mult{h}{v_0} + f$ is in the interior of $\Cp$.
    By Assumption \ref{as:x-full}, there exists $w \in K^*$ such that $b = \mult{h}{w} \ge b_{\min} > 0$. 
    Since $f_i$ is bounded on $S_i$, we can find $\eta \in \RR_+$ such that 
    $$
    \eta (\mult{h}{w})_i + f_i > 0 \ \ \ \forall \ i = 1,\ldots, m,
    $$
    implying 
    that $v_0 = \eta w$ is the required vector. Thus, \Cref{con:slaters} holds for $-\dop^*$, implying $-\dop^* = -\pop^*$, or equivalently, $\dop^* = \pop^*$. \QED

\section{Dual of bounded functionals}\label{app:lem}
\begin{lemma}\label{lem:dual_analytic:app}
    The space $(\drng^{\bnd}_i)^*$ equipped with the norm $\norm{\cdot}^*_i$ is isometrically isomorphic to the space of real-analytic functions $A$ with norm $\norm{f}_A \BYDEF \sum_{\a \in \mathbb{N}^{n_i}} \a! | f_\a |$. 
\end{lemma}
\begin{proof}
    The result follows readily by verifying the following mapping is an isometric isomorphism: $$\ph : (\cl D_i^{\bnd})^* \rightarrow A$$
    $$
    \ph(f)(\bm x) \BYDEF \sum_{\a \in \mathbb{N}^{n_i}} f(\bm d_\a) \bm x^\a,
    $$
    where $\bm d_\a(\bm x^\b) = 1$ if $\a = \b$ and $0$ otherwise.\\
    \textit{1. Real-analyticity of $\ph(f)$.} Let $f \in (\cl D_i^{\bnd})^*$. Then $$|f(\bm d_{\alpha})| \le \norm{f}^*_i \norm{\bm d_{\alpha}}_i = \norm{f}^*_i \sup_{\beta \in \mathbb{N}^{n_i}}|\bm d_{\alpha}(\bm x^{\beta})| / \beta ! = \norm{f}^*_i / \alpha !. $$
    The absolute value of the $\a$-th term in the series of $\ph(f)$ is $|f(\bm d_\a) \bm x^\a| = |f(\bm d_\a)| |\bm x^\a| \le (\norm{f}_i^* / \a!) |\bm x^\a|$
    and the $\a!$ in the denominator ensures convergence. That $\ph(f)$ also belongs to the space $A$ (i.e., has a finite $A$-norm) is established by the isometry in point 5. \\
    \textit{2. Linearity.} Let $f, g \in (\cl D_i^{\bnd})^*$ and $a, b \in \mathbb{R}$. The fact $\ph(af + bg) = a \ph(f) + b \ph(g)$ follows readily from the definition of $\ph$.\\
    \textit{3. Injectivity.} Let $\ph(f) = \ph(g)$. By linearity, this is equivalent to showing that if $\ph(F)=0$ for $F=f-g$, then $F=0$.
    The condition $\ph(F)=0$ means that 
    $$\sum_{\a \in \mathbb{N}^{n_i}} F(\bm d_\a) \bm x^\a = 0.$$  This implies that all coefficients are zero: $F(\bm d_\a) = 0$ for all $\a \in \mathbb{N}^{n_i}$.
    From the definition of the $A$-norm, this gives $\norm{\ph(F)}_A = \sum_{\a \in \mathbb{N}^{n_i}} \a!|F(\bm d_\a)| = 0$.
    The norm-preservation property, established in point 5 below, shows that $\norm{F}_i^* = \norm{\ph(F)}_A$.
    Therefore, $\norm{F}_i^* = 0$, implying $F=0$, and subsequently $f=g$. Thus, $\ph$ is injective.\\
    \textit{4. Surjectivity.} Let $h = \sum_{\a \in \mathbb{N}^{n_i}} h_\a \bm x_i^\a \in A$. Take $$f_h(\lambda) \BYDEF \sum_{\a \in \mathbb{N}^{n_i}} h_\a \lambda(\bm x_i^\a).$$ We verify that $f_h$ is well-defined and continuous. Note that the above series converges absolutely for any $\lambda \in \cl D_i^{\bnd}$:
    $$
    |f_h(\lambda)| \le \sum_{\a \in \mathbb{N}^{n_i}} |h_\a| |\lambda(\bm x_i^\a)| \le \sum_{\a \in \mathbb{N}^{n_i}} |h_\a| (\norm{\lambda}_i \a!) = \norm{\lambda}_i \left( \sum_{\a \in \mathbb{N}^{n_i}} \a! |h_\a| \right) = \norm{\lambda}_i \norm{h}_A.
    $$
    This inequality also shows that $f_h$ is a bounded linear functional with $\norm{f_h}_i^* \le \norm{h}_A$, so $f_h \in (\cl D_i^{\bnd})^*$. We may then apply the definition of $\ph(\cdot)$:
    $$
    \ph(f_h) = \sum_{\a \in \mathbb{N}^{n_i}} h_\a \bm x_i^\a = h,
    $$
    as required.\\
    \textit{5. Norm preservation.} We show:
    \begin{enumerate}[label=(\alph*)]
        \item $\norm{f}_i^* \le \sum_{\a \in \mathbb{N}^{n_i}} \a ! | f(\bm d_\a) |$,
        \item $\norm{f}_i^* \ge \sum_{\a \in \mathbb{N}^{n_i}} \a ! | f(\bm d_\a) |$, and
        \item $\| \ph(f) \|_{A} = \sum_{\a \in \mathbb{N}^{n_i}} \a ! | f(\bm d_\a) |$. 
    \end{enumerate}
    (a) Let $f \in (\cl D_i^{\bnd})^*$ and $\lambda \in \cl D_i^{\bnd}$ with $\norm{\lambda}_i \le 1$. Then,
    $$
    |\apply{f}{\lambda}| = \left| \sum_{\a \in \mathbb{N}^{n_i}} f(\bm d_\a) \lambda(\bm x_i^\a) \right| \le \sum_{\a \in \mathbb{N}^{n_i}} |f(\bm d_\a)| |\lambda(\bm x_i^\a)|.
    $$
    Since $\norm{\lambda}_i \le 1$, we have $|\lambda(\bm x_i^\a)| \le \a! \norm{\lambda}_i \le \a!$. Thus,
    $$
    |\apply{f}{\lambda}| \le \sum_{\a \in \mathbb{N}^{n_i}} |f(\bm d_\a)| \a! = \norm{\ph(f)}_A.
    $$
    Taking the supremum over $\lambda$ with $\norm{\lambda}_i \le 1$ yields $\norm{f}_i^* \le \norm{\ph(f)}_A$. \\
    (b) Given a non-zero $f$, define $\lambda_N \in \drng_i^\bnd$ as the finite sum:
    $$
    \lambda_N \BYDEF \sum_{|\a| \le N} \a! \, \text{sgn}(f(\bm d_\a))\bm d_\a.
    $$
    The norm of this functional is
    $$
    \norm{\lambda_N}_i = \sup_{\b \in \mathbb{N}^{n_i}} \frac{|\apply{\lambda_N}{\bm x_i^\b}|}{\b!} = \sup_{|\b| \le N} \frac{|\b! \, \text{sgn}(f(\bm d_\b))|}{\b!} = 1.
    $$
    Since $\norm{\lambda_N}_i = 1$, by the definition of the operator norm, we have $\norm{f}_i^* \ge |f(\lambda_N)|$ for any $N$. Evaluating $|f(\lambda_N)|$  gives
    $$
    |f(\lambda_N)| = \left| \sum_{|\a| \le N} \a! \, \text{sgn}(f(\bm d_\a)) f(\bm d_\a) \right| = \sum_{|\a| \le N} \a! |f(\bm d_\a)|.
    $$
    This inequality holds for all $N$. Taking the limit as $N \to \infty$, we obtain the desired bound:
    $$
    \norm{f}_i^* \ge \lim_{N\to\infty} \sum_{|\a| \le N} \a! |f(\bm d_\a)| = \sum_{\a \in \mathbb{N}^{n_i}} \a! |f(\bm d_\a)|.
    $$
    (c) We have 
    $$\norm{\ph(f)}_A = \norm{\sum_{\a \in \mathbb{N}^{n_i}} f(\bm d_\a) \bm x^\a}_A   $$
    $$=  \sum_{\a \in \mathbb{N}^{n_i}} \a! |f(\bm d_\a)|, $$
    as required.
Thus, $\ph$ is an isometric isomorphism.
\end{proof}

\end{document}